\documentclass[a4paper,twoside,10pt]{article}

\usepackage[a4paper,left=3cm,right=3cm, top=3cm, bottom=3cm]{geometry}
\usepackage[latin1]{inputenc}
\usepackage{mathrsfs}
\usepackage{graphicx}
\usepackage{epstopdf}
\usepackage{amsmath}
\usepackage{amsthm}
\usepackage{amssymb}
\usepackage{hyperref}
\usepackage{stmaryrd}
\usepackage{float}
\usepackage{calrsfs}
\DeclareMathAlphabet{\pazocal}{OMS}{zplm}{m}{n}
\usepackage{bigints}
\usepackage{cite}
\usepackage{color}
\usepackage[abs]{overpic}
\usepackage[font=footnotesize,labelfont=bf]{caption}
\usepackage{cases}
\usepackage{tikz}
\usepackage{algorithmic}
\usepackage{rotating}
\usepackage{blkarray}
\usetikzlibrary{matrix,calc,arrows}
\usepackage[author={Lorenzo}]{pdfcomment}
\usepackage{soul,xcolor}
\usepackage{verbatim}
\usepackage{graphicx}
\usepackage{subcaption}
\usepackage{algorithm}

\DeclareFontFamily{OT1}{pzc}{}
\DeclareFontShape{OT1}{pzc}{m}{it}{<-> s * [1.10] pzcmi7t}{}
\DeclareMathAlphabet{\mathpzc}{OT1}{pzc}{m}{it}
\DeclareMathOperator{\diam}{diam}

\restylefloat{table}
\theoremstyle{plain}
\newtheorem{thm}{Theorem}[section]

\newtheorem{lem}[thm]{Lemma}

\theoremstyle{definition}

\theoremstyle{remark}
\newtheorem{remark}{Remark}

\newcommand{\C}{\mathbb{C}}
\newcommand{\R}{\mathbb{R}}
\newcommand{\N}{\mathbb{N}}

\newcommand{\red}[1]{{\color{red}{#1}}} 
\newcommand{\blue}[1]{{\color{blue}{#1}}} 
\newcommand{\orange}[1]{{\color{orange}{#1}}} 

\newcommand{\im}{\textup{i}} 
\newcommand{\E}{K} 
\newcommand{\taun}{\mathcal T_n} 
\newcommand{\e}{e} 
\newcommand{\PiEpqtilde}{\Pi_{\p^\E,\ptilde^\E}^{(2),\E}} 
\newcommand{\PiEp}{\Pi_{\p^\E}^{(1),\E}} 
\newcommand{\PiE}{\Pi^{\E}} 
\newcommand{\Pizepqtilde}{\Pi_{\pe}^{0,\e}} 
\newcommand{\SE}{S^\E} 
\newcommand{\un}{u_h} 
\newcommand{\vn}{v_h} 
\newcommand{\uh}{u_h} 
\newcommand{\vh}{v_h} 
\newcommand{\w}{w} 
\renewcommand{\k}{k} 
\newcommand{\n}{\mathbf n} 
\renewcommand{\a}{a} 
\renewcommand{\b}{b} 
\newcommand{\aE}{\a^\E} 
\newcommand{\ahE}{\aE_\h} 
\newcommand{\ch}{c_\h^{\partial \Omega}} 
\newcommand{\bh}{b_\h} 
\newcommand{\Fh}{F_\h} 
\newcommand{\PW}{\mathbb {PW}} 
\newcommand{\EW}{\mathbb {EW}} 
\newcommand{\Vh}{V_{\h}} 
\newcommand{\VhE}{V_\h(\E)} 
\newcommand{\wjEVE}{w^{\text{EV},\E}_j} 
\newcommand{\wlehat}{\widehat{w}_\ell^e} 
\newcommand{\pehat}{\widehat{p}_\e} 
\renewcommand{\d}{\textbf{\textup{d}}} 
\newcommand{\h}{h} 
\newcommand{\hE}{\h_\E} 
\newcommand{\he}{\h_e} 
\newcommand{\p}{p} 
\newcommand{\q}{q} 
\newcommand{\qE}{\q_\E} 
\newcommand{\ptilde}{\widetilde \p} 
\newcommand{\qtilde}{\widetilde \q} 
\newcommand{\pe}{p_e} 
\newcommand{\pE}{p_\E} 
\newcommand{\g}{g} 
\newcommand{\En}{\mathcal E_n} 
\newcommand{\Enb}{\mathcal E_n^B} 
\newcommand{\EnI}{\mathcal E_n^I} 
\newcommand{\EnIone}{\mathcal E_n^{1,I}} 
\newcommand{\EnItwo}{\mathcal E_n^{2,I}} 
\newcommand{\EnBone}{\mathcal E_n^{1,B}} 
\newcommand{\EnBtwo}{\mathcal E_n^{2,B}} 
\newcommand{\EnGamma}{\mathcal E_n^{\Gamma}} 
\newcommand{\EE}{\mathcal E^{\E}} 
\newcommand{\x}{\textbf{\textup{x}}} 
\newcommand{\xE}{\textbf{\textup{x}}_\E} 
\newcommand{\dl}{\textbf{\textup{d}}_\ell} 
\newcommand{\elltilde}{\widetilde \ell} 

\newcommand{\card}{\textup{card}} 

\newcommand{\ds}{\text{d}s}
\newcommand{\dx}{\text{d}x}
\newcommand{\xbold}{\mathbf x} 
\newcommand{\kcal}{\mathfrak k} 
\newcommand{\thetacrit}{\theta_{\text{crit}}} 
\newcommand{\thetainc}{\theta_{\text{inc}}} 
\newcommand{\thetaT}{\theta_{\text{T}}} 
\newcommand{\thetaR}{\theta_{\text{R}}} 
\renewcommand{\span}{\text{span}} 
\newcommand{\thetaEW}{\theta^{\text{EW}}} 

\newcommand{\dof}{\textup{dof}} 

\newcommand{\nE}{\textbf{\textup{n}}_\E} 
\newcommand{\Ne}{n_\E} 

\newcommand{\we}{w^\e} 

\newcommand{\dir}{\textbf{\textup{d}}} 

\newcommand{\AboldEone}{\mathbf {A}^{(1),\E}}
\newcommand{\SboldEone}{\mathbf {S}^{(1),\E}}
\newcommand{\GboldEone}{\mathbf {G}^{(1),\E}}
\newcommand{\PiboldEsone}{\mathbf {\Pi}^{(1),\E}_*}
\newcommand{\PiboldEone}{\mathbf {\Pi}^{(1),\E}}
\newcommand{\DboldEone}{\mathbf {D}^{(1),\E}}
\newcommand{\BboldEone}{\mathbf {B}^{(1),\E}}
\newcommand{\IboldEone}{\mathbf {I}^{(1),\E}}
\newcommand{\AboldEtwo}{\mathbf {A}^{(2),\E}}

\newcommand{\GboldEtwo}{\mathbf {G}^{(2),\E}}
\newcommand{\GboldEtwohat}{\mathbf {G}^{(2),\E}}

\newcommand{\DboldEtwo}{\mathbf {D}^{(2),\E}}
\newcommand{\BboldEtwo}{\mathbf {B}^{(2),\E}}

\newcommand{\Rbolde}{\mathbf {R}^\e}
\newcommand{\Gbolde}{\mathbf {G}^\e_0}
\newcommand{\Bbolde}{\mathbf {B}^\e_0}
\newcommand{\pbold}{\mathbf \p}
\newcommand{\pboldo}{\pbold^{(1)}}
\newcommand{\pboldt}{\pbold^{(2)}}
\newcommand{\pboldtilde}{\widetilde \pbold}
\newcommand{\pboldepsilon}{\pbold_{\mathcal E_n}}

\newcommand{\phE}{p^{\E}}
\newcommand{\phEhat}{p^{\E}}
\newcommand{\phatE}{\widehat{p}_{\E}}
\newcommand{\ptildeE}{\ptilde^{\E}}

\newcommand{\PWtilde}{\widetilde{\PW}_{p_e}} 
\newcommand{\PWtildepqtilde}{\widetilde{\PW}^{(2)}_{\p^{\E},\ptilde^{\E}}} 
\newcommand{\PWtildepqtildeminus}{\widetilde{\PW}^{(2)}_{\p^{\E^-},\ptilde^{\E^-}}} 
\newcommand{\PWtildepqtildeplus}{\widetilde{\PW}^{(2)}_{\p^{\E^+},\ptilde^{\E^+}}} 
\newcommand{\uinc}{u_{\textup{inc}}}

\begin{tiny}
\author{
\normalsize{
}}
\end{tiny}

\date{}

\title{\Large{Extension of the nonconforming Trefftz virtual element method to the Helmholtz problem with piecewise constant wave number}}
\date{}
\author{Lorenzo Mascotto\thanks{Faculty of Mathematics, University of Vienna, 1090 Vienna, Austria (lorenzo.mascotto@univie.ac.at, alex.pichler@univie.ac.at)},\ Alexander Pichler\footnotemark[1]}

\begin{document}
\maketitle

\begin{abstract}
\noindent
We extend the nonconforming Trefftz virtual element method introduced in~\cite{ncTVEM_Helmholtz} to the case of the fluid-fluid interface problem, that is, a Helmholtz problem with piecewise constant wave number.
With respect to the original approach, we address two additional issues:
firstly, we define the coupling of local approximation spaces with piecewise constant wave numbers;
secondly, we enrich such local spaces  with special functions capturing the physical behaviour of the solution to the target problem.
As these two issues are directly related to an increase of the number of degrees of freedom, we use a reduction strategy inspired by~\cite{TVEM_Helmholtz_num},
which allows to mitigate the growth of the dimension of the approximation space when considering~$\h$- and $\p$-refinements.
This renders the new method highly competitive in comparison to other Trefftz and quasi-Trefftz technologies tailored for the Helmholtz problem with piecewise constant wave number.
A wide range of numerical experiments, including the~$\p$-version with quasi-uniform meshes and the $\h\p$-version with isotropic and anisotropic mesh refinements, is presented.


\end{abstract}

\section{Introduction} \label{section introduction}
Efficient methods for the approximation of solutions to high frequency wave propagation problems have received an increasing attention over the last two decades.
Starting from the ultra weak variational formulation of Cessenat and Despr\'es~\cite{cessenatdespres_basic}, many wave based methods for the Helmholtz problem have been introduced and analyzed, see~\cite{PWDE_survey} for an overview of the topic.
Such methods are in general based on trial and test spaces consisting of piecewise (discontinuous) plane waves.

In the framework of the virtual element method (VEM)~\cite{VEMvolley, hitchhikersguideVEM}, which can be seen an extension of the finite element method (FEM) to polytopal meshes and as the ultimate evolution of the mimetic finite differences~\cite{BLM_MFD, lipnikov2014mimetic},
an $H^1$-conforming method for the Helmholtz problem was introduced in~\cite{Helmholtz-VEM}.
Such a method, known as the plane wave VEM, is based on local approximation spaces containing plane waves that are eventually patched continuously with the aid of a partition of unity,
in the spirit of the pioneering work of Melenk and Babu\v ska~\cite{BabuskaMelenk_PUMintro}.

More recently, a novel nonconforming Trefftz-VEM for the Helmholtz problem was developed in~\cite{ncTVEM_Helmholtz} as an extension of the harmonic VEM~\cite{conformingHarmonicVEM, ncHVEM} for the Laplace problem.
The two main features of this method are $(i)$ that it is \emph{Trefftz} (i.e., local spaces consist of functions belonging to the kernel of the target differential operator)
and $(ii)$ that it falls within the nonconforming virtual element framework, see e.g.~\cite{nonconformingVEMbasic, VEM_fullync_biharmonic, cangianimanzinisutton_VEMconformingandnonconforming, gardini2018nonconforming}.
Although in the basic construction of the method more degrees of freedom than e.g. in the plane wave discontinuous Galerkin method~\cite{GHP_PWDGFEM_hversion} are needed,
a modification of a strategy introduced in~\cite{TVEM_Helmholtz_num} allows to significantly reduce the dimension of the approximation space as well as the condition number of the resulting final system;
this renders the nonconforming Trefftz-VEM approach highly competitive in comparison with other Trefftz technologies.
Roughly speaking, the main idea of this strategy is that, whenever two basis functions are generating ``almost'' the same space, one of the two can be kicked out from the set of basis functions,
yet not jeopardizing the approximation properties of the space.
\medskip

The methods described so far have been tailored for the simplest Helmholtz problem, that is, for problems with constant wave number; the case of variable wave number is more challenging and intriguing.
The instance of analytic wave number was faced in a number of works, for instance by Imbert-G\'erard and collaborators in~\cite{DespresImbert2014, ImbertGerard2015, imbert2015well, imbertMonk2017}, where the so-called \emph{generalized plane waves} were introduced;
the idea behind that approach is to employ approximation spaces that are globally discontinuous and locally spanned by combinations of exponential functions applied to complex polynomials.
It is worthwhile to notice that this method is \emph{quasi-Trefftz} only (that is, when applying the Helmholtz operator to the basis functions, one gets a quantity which is converging to zero as the mesh size decreases and the dimension of the local space increases)
and that it generalizes the discontinuous enrichment method~\cite{tezaur2014airy}, which addresses the simpler case of linear wave number.
Another quasi-Trefftz method for smooth wave numbers is provided in a work of Betcke and Phillips in~\cite{betcke2011adaptive}; there, the basis functions are modulated plane waves, i.e., products of plane waves with polynomials.

On the other hand, the instance of piecewise constant wave numbers gives raise to the fluid-fluid interface problem, which models the transmission of a wave between two fluids with different refraction indices; such model is in fact the one tackled in the present paper.
We mention that the plane wave discontinuous Galerkin method and the discontinuous enrichment method have been successfully applied to this problem, see~\cite{luostari2013improvements} and~\cite{tezaur2008discontinuous}, respectively.
In those two approaches, Bessel functions were employed in addition to plane waves, and other special functions (namely evanescent waves) were added to capture the physical behaviour of the solution at the interface between the two fluids.

\medskip
In this paper,
\begin{enumerate}
\item we extend the nonconforming Trefftz-VEM of~\cite{ncTVEM_Helmholtz, TVEM_Helmholtz_num} to the case of piecewise constant wave numbers, and
\item following what was done in~\cite{luostari2013improvements, tezaur2008discontinuous}, we also include proper special functions in the approximation spaces to capture the behaviour of the physical solution.
\end{enumerate}
We will see that both issues elegantly fit within the nonconforming Trefftz-VEM framework.
Further, by employing the removing technique introduced in~\cite{TVEM_Helmholtz_num}, an extremely robust numerical performance is obtained.

The method we are going to present is characterized by local spaces containing plane (and possibly evanescent) waves, plus additional functions implicitly defined as solutions to local Helmholtz problems
with impedance boundary conditions in proper 1D plane and evanescent wave spaces.
These local spaces are eventually coupled in a nonconforming fashion \emph{\`a la} Crouzeix-Raviart
(in the sense that the jumps across the interface between elements have zero moments up to a certain order).
The fact that the functions in the approximation space are unknown in closed form entails that, in order to implement the method, one can not use the continuous sesquilinear form;
rather, discrete counterparts based on projections onto (plane and evanescent) wave spaces and stabilizing sesquilinear forms are employed.

The outline of the paper is as follows.
Section~\ref{section fluid-fluid interface problem} is devoted to the description of the model problem, whereas
Section~\ref{section functional spaces} provides the notation for plane wave and evanescent wave spaces, as well as for nonconforming Sobolev spaces.
The method, including the definition of the local and the global spaces, of a set of degrees of freedom, of suitable projections onto wave spaces, and of suitable stabilizations, is the topic of Section~\ref{section ncTVEM fluidfluid}.
In Section~\ref{section numerical results}, we briefly discuss the implementation details of the method and we present a number of numerical experiments.
In particular, we study the performance of the $\h$- and of $\p$-versions, whenever the meshes are conforming with respect to the interface between the two fluids (i.e., the wave number is piecewise constant over the polygonal decomposition);
the rate of convergence is algebraic and exponential in terms of~$\h$ and~$\p$ in the former and in the latter case, respectively.
Another interesting set of experiments is focused on testing the robustness of the method, whenever some elements are cut by the interface;
on such elements, in fact, the solution to the fluid-fluid problem has typically a very low Sobolev regularity, and therefore the convergence of the $\h$- and of $\p$-versions is poor.
Consequently, the $\h\p$-version with geometric isotropic and anisotropic mesh refinements is employed, leading to algebraic and exponential convergence
in terms of proper roots of the number of degrees of freedom in the former and in the latter case, respectively.
Some conclusions are stated in Section~\ref{section conclusions}.
It is important to highlight that in the implementation of the method, quadrature formulas are needed only for the approximation of the terms involving the boundary data.

We stress that, although the present paper is aimed at the approximation of the Helmholtz problem with piecewise constant wave number solely, the setting of the nonconforming Trefftz-VEM can be applied in other situations.
For instance, one could extend the method to the case of analytic wave number, dovetailing the nonconforming VEM technology with the tools stemming from the theory of generalized plane waves.
A possible advantage of employing a variant of the approach presented herein, in lieu of the  discontinuous Galerkin one~\cite{imbertMonk2017}, is that the or\-tho\-go\-na\-li\-za\-tion-and-fil\-te\-ring technique inspired by~\cite{TVEM_Helmholtz_num}
could lead to an improved convergence rate in terms of the number of degrees of freedom and to an improved conditioning of the final system.

As a final comment, we stress that another appealing feature of the nonconforming setting is that the extension to the 3D case
is much more straightforward than in the $H^1$-conforming setting; see \cite[Section 3.7]{ncHVEM} for a description of such an extension in the case where the target differential operator is the Laplacian.


\paragraph*{Notation.}
Throughout the paper, we will employ the standard notation for Sobolev spaces, norms, seminorms and inner products, see e.g.\cite{adamsfournier}.
More precisely, given a domain~$D\subset \mathbb R^2$, we denote  by~$H^s(D)$ the Sobolev space of functions with square integrable weak derivatives up to order~$s$, for some nonnegative integer~$s$, over~$D$,
and the corresponding seminorms and norms by~$|\cdot|_{s,D}$ and~$\lVert \cdot \rVert_{s,D}$, respectively.
Sobolev spaces of noninteger order can be defined by interpolation theory.
If the domain~$D$ is also bounded,~$H^{1/2}(\partial D)$ denotes the space of the traces of~$H^1(D)$ functions and~$H^{-1/2}(\partial D)$ denotes its dual space. Further, $(\cdot,\cdot)_{0,D}$ is the usual~$L^2$ inner
product over~$D$. Lastly, we denote by~$\N_0$ the set of all natural numbers including~$0$, and by~$\N_{\ge r}$, for some~$r>0$, the set of all natural numbers larger than or equal to~$r$.

\section{The fluid-fluid interface problem} \label{section fluid-fluid interface problem}
Given a polygonal domain~$\Omega \subset \mathbb R^2$, a piecewise (real-valued) constant wave number~$\kcal\in L^{\infty}(\Omega)$, and~$\g \in H^{-\frac{1}{2}}(\partial \Omega)$, we aim to approximate the solution to the problem
\begin{equation} \label{strong monolithic formulation}
\left\{
\begin{alignedat}{2}
-\Delta u - \kcal ^2 u &= 0 && \quad \text{in } \Omega\\
\nabla u \cdot \n_{\Omega} + \im \kcal u &=\g && \quad \text{on } \partial \Omega,\\
\end{alignedat}
\right.
\end{equation}
where~$\n_{\Omega}$ denotes the unit normal vector on~$\partial \Omega$ pointing outside~$\Omega$ and~$\im$ is the imaginary unit.

The corresponding weak formulation to problem~\eqref{strong monolithic formulation} reads
\begin{equation} \label{weak monolithic formulation}
\begin{cases}
\text{find } u \in V:=H^1(\Omega) \, \text{such that}\\
b(u,v) = \langle \g,v \rangle \quad \forall v \in V,
\end{cases}
\end{equation}
where the sesquilinear form~$b(\cdot,\cdot)$ is given by
\begin{equation} \label{complete form}
\b(u,v) := a(u,v) + \im (\kcal  u,v)_{0,\partial \Omega}
\end{equation}
with
\begin{equation}
a(u,v) := \int_\Omega \nabla u \cdot \overline{\nabla v } \, \dx- \int_\Omega \kcal u \overline v \, \dx,
\end{equation}
and the right-hand side is defined as
\begin{equation} \label{rhs}
\langle \g,v \rangle :=\int_{\partial \Omega} g \overline{v} \, \ds.
\end{equation}
The well-posedness of the problem~\eqref{weak monolithic formulation} can be proven as in e.g.~\cite[Theorem 2.4]{graham2018stability}.

For the sake of simplicity, we will assume in the following that the domain~$\Omega=(-1,1)^2$ is split into two parts~$\Omega_1:=(-1,1) \times (-1,0)$ and~$\Omega_2:=(-1,1) \times (0,1)$,
and that the wave number~$\kcal$ is piecewise constant over~$\Omega_1$ and~$\Omega_2$; more precisely, we set~$k_i:=\kcal _{|{\Omega_i}} = n_i \k$, $i=1,2$, where $k \in \R$, and~$n_1$, $n_2 \in \R$ with~$n_1>n_2$ are the so-called \textit{refraction indices}, respectively.
The more general situation with multiple refraction indices and subdomains is a straightforward modification of the case with two subdomains.

Denoting by~$\Gamma:=\partial \Omega_1\cap \partial \Omega_2$ the interface between the two subdomains with fixed unit normal vector~$\n_{\Gamma}$, problem~\eqref{strong monolithic formulation} can be reformulated as the transmission problem
\begin{equation} \label{strong transmission problem}
\left\{
\begin{alignedat}{2}
\text{find $u_i\in H^1(\Omega_i)$, $i=1,2$, such that} \hspace{-3cm} &&\\
-\Delta u_i - k_i^2 u_i &= 0 && \quad \text{in } \Omega_i,\quad i=1,2\\
\nabla u_i \cdot \n_{\Omega_i} + \im k_i u_i &= g &&\quad  \text{on }\partial \Omega_i \setminus \Gamma,\quad i=1,2\\
u_1 &= u_2 && \quad \text{on } \Gamma\\
\nabla u_1 \cdot \n_{\Gamma} &= \nabla u_2 \cdot \n_{\Gamma} && \quad\text{on } \Gamma.\\
\end{alignedat}
\right.
\end{equation}
This model goes under the name of \textit{fluid-fluid interface problem}.
From a physical standpoint, it describes the propagation of waves through a domain split into two subdomains containing different fluids (e.g. water-air).
Typically, some reflection/transmission phenomenon occurs at the interface~$\Gamma$.
For instance, assuming that there is an incoming traveling plane wave in~$\Omega_1$ with incident angle~$\thetainc$ formed by the  direction of the incoming wave with the interface~$\Gamma$, the model describes the propagation of such wave from~$\Omega_1$ to~$\Omega_2$.
Depending on the angle~$\thetainc$, a different behaviour may occur in~$\Omega_1$ and~$\Omega_2$.

In order to describe the two possible outcomes, we introduce the so-called \textit{critical angle}
\begin{equation} \label{critical angle}
\thetacrit := \cos ^{-1} \left(\frac{n_2}{n_1}\right).
\end{equation}
If~$\thetainc \ge \thetacrit$, the incoming wave is partially refracted at~$\Gamma$ with angle~$\theta_R$ (having the same measure as~$\thetainc$) and transmitted in the subdomain~$\Omega_2$ with transmission angle~$\thetaT$,
which is computed by means of Snell's law
\begin{equation*}
n_1 \cos(\thetainc)=n_2 \cos(\thetaT).
\end{equation*}
Otherwise, if~$\thetainc < \thetacrit$, the incoming wave is totally refracted (with angle~$\thetaR$, having again the same measure as~$\thetainc$); however, in the subdomain~$\Omega_2$ some evanescent modes,
exponentially decaying in terms of the distance from the interface~$\Gamma$, appear.
This phenomenon is known in the literature as \textit{total internal reflection}. 
In Figure~\ref{figure physical model}, the two different situations depending on the choice of~$\thetainc$ are depicted.

\begin{center}
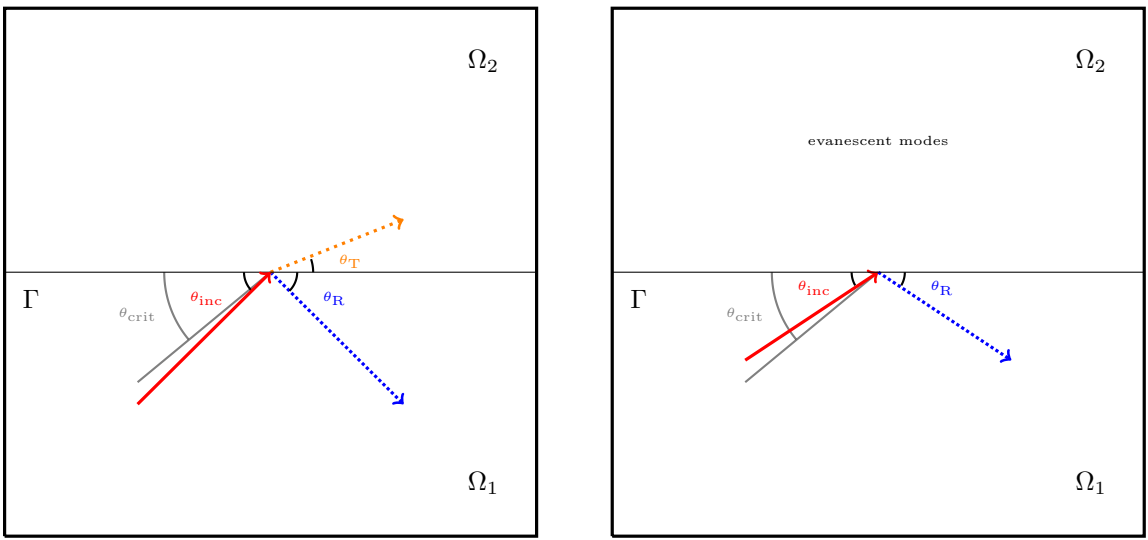
\begin{figure}[H]
\begin{minipage}{0.30\textwidth}
\begin{center}
\begin{tikzpicture}[scale=7]
\draw[blue,densely dotted,very thick,->]  (0.5,0.5) -- (0.75,0.25); \draw[orange, dotted, very thick,->]  (0.5,0.5) -- (0.75,0.6);
\draw [black,thick,domain=180:225] plot ({0.5+0.05*cos(\x)}, {0.5+0.05*sin(\x)}); \draw(0.38, 0.45) node[red] {\tiny{$\thetainc$}};
\draw [black,thick,domain=315:360] plot ({0.5+0.05*cos(\x)}, {0.5+0.05*sin(\x)}); \draw(0.62, 0.45) node[blue] {\tiny{$\thetaR$}};
\draw [black,thick,domain=0:18] plot ({0.5+0.08*cos(\x)}, {0.5+0.08*sin(\x)}); \draw(0.65, 0.52) node[orange] {\tiny{$\thetaT$}};
\draw[black, thick,-, opacity=0.5]  (0.25,7/6*0.25) -- (0.5,0.5); 
\draw[black, very thick, -] (0,0) -- (1,0) -- (1,1) -- (0, 1) -- (0,0);
\draw[black,  -] (0,0.5) -- (1,0.5);
 \draw(0.05, 0.45) node[black] {$\Gamma$};
\draw(0.9, 0.9) node[black] {$\Omega_2$}; \draw(0.9, 0.1) node[black] {$\Omega_1$};
\draw[red,  very thick,->]  (0.25,0.25) -- (0.5,0.5);
\draw [black,thick,domain=180:219.5, opacity=0.5] plot ({0.5+0.2*cos(\x)}, {0.5+0.2*sin(\x)}); \draw(0.25, 0.42) node[black, opacity=0.5] {\tiny{$\thetacrit$}};
\end{tikzpicture}
\end{center}
\end{minipage}
\quad\quad\quad\quad\quad\quad\quad\quad\quad\;
\begin{minipage}{0.30\textwidth}
\begin{center}
\begin{tikzpicture}[scale=7]
\draw [black,thick,domain=180:210] plot ({0.5+0.05*cos(\x)}, {0.5+0.05*sin(\x)}); \draw(0.38, 0.47) node[red] {\tiny{$\thetainc$}};
\draw [black,thick,domain=330:360] plot ({0.5+0.05*cos(\x)}, {0.5+0.05*sin(\x)}); \draw(0.62, 0.47) node[blue] {\tiny{$\thetaR$}};
\draw(0.5, 0.75) node[black] {\tiny{\text{evanescent modes}}};
\draw[black, thick,-, opacity=0.5]  (0.25,7/6*0.25) -- (0.5,0.5); 
\draw [black,thick,domain=180:219.5, opacity=0.5] plot ({0.5+0.2*cos(\x)}, {0.5+0.2*sin(\x)}); \draw(0.25, 0.42) node[black, opacity=0.5] {\tiny{$\thetacrit$}};
\draw[black, very thick, -] (0,0) -- (1,0) -- (1,1) -- (0, 1) -- (0,0);
 \draw(0.05, 0.45) node[black] {$\Gamma$};
\draw(0.9, 0.9) node[black] {$\Omega_2$}; \draw(0.9, 0.1) node[black] {$\Omega_1$};
\draw[red,  very thick,->]  (0.25,4/3 * 0.25) -- (0.5,0.5); \draw[blue,densely dotted,very thick,->]  (0.5,0.5) -- (0.75,4/3*0.25);
\draw[black,  -] (0,0.5) -- (1,0.5);
\end{tikzpicture}
\end{center}
\end{minipage}
\caption{\textit{Left:} $\thetainc \ge \thetacrit$. The incoming wave is partially refracted at~$\Gamma$ and partially transmitted in form of a plane wave with direction given by the angle~$\thetaT$ in~$\Omega_2$.
\textit{Right:} $\thetainc < \thetacrit$. The incoming wave is totally refracted; only evanescent modes appear in~$\Omega_2$.
\textit{Legend:} the directions of the incident, the reflected, and the transmitted plane waves are \red{straight red}, \blue{dashed blue}, and \orange{dotted orange}, respectively. The critical angle~$\thetacrit$ is depicted in grey.}
\label{figure physical model}
\end{figure}
\end{center}
A couple of explicit solutions to the problem~\eqref{weak monolithic formulation} in the transmission and the total internal reflection cases are described in Sections~\ref{subsubsection testcase1} and~\ref{subsubsection testcase2}, respectively.

\section{Plane waves, evanescent waves, and nonconforming Sobolev spaces} \label{section functional spaces}
In this section, we first define the spaces of plane waves and evanescent waves over elements and edges, and, subsequently, we construct a class of nonconforming Sobolev spaces.

We will introduce two types of local spaces, namely plane wave based spaces over the elements in~$\Omega_1$ and spaces based on both plane waves and evanescent waves over the elements contained in~$\Omega_2$.
The choice for the latter spaces is inspired by~\cite{tezaur2008discontinuous, luostari2013improvements}, where evanescent waves were added as special functions to the standard plane wave and Bessel spaces, respectively,
to capture the evanescent modes occurring in specific situations described in Section~\ref{section fluid-fluid interface problem}. We anticipate that variants of such spaces are possible and will be discussed in Section~\ref{section numerical results}. 

First, we fix some notation. Given~$\taun^1$ and~$\taun^2$ two decompositions into polygons of~$\Omega_1$ and~$\Omega_2$, respectively,
then $\taun:=\taun^1 \cup \taun^2$ is a decomposition of~$\Omega$ into polygons.
Further, for all~$\E\in \taun$, we denote by~$\xE$ its barycenter, by~$\hE:=\diam(\E)$ its diameter, and by~$\h:=\max_{\E \in \taun} \hE$ the mesh size of~$\taun$.

Moreover, we write~$\EnIone$ and~$\EnBone$ for the sets of interior edges in~$\taun^1$, and boundary edges in~$\taun^1$ not belonging to~$\Gamma$, respectively. Similarly, we introduce the sets~$\EnItwo$ and~$\EnBtwo$ for~$\taun^2$.
The symbol~$\EnGamma$ denotes the set of edges of~$\taun$ on~$\Gamma$. Further, we define $\EnI:=\EnIone \cup \EnItwo$ and $\Enb:=\EnBone \cup \EnBtwo$.
Finally, $\he$ denotes the length of a given edge~$\e \in \En$, with~$\En$ denoting the set of all edges of~$\taun$, and~$\Ne$ is the number of edges of a given polygon~$\E \in \taun$.
\medskip

Having this, we introduce the local plane wave spaces over the elements in~$\taun^1$. To this purpose, given~$\E \in \taun^1$, let $\{\dl^{\E}\}_{\ell=1}^{\p^{\E}}$, $\p^{\E}=2\q^{\E}+1$, $\q^{\E} \in \mathbb N$, be a bunch of equidistributed normalized directions.
Then, denoting by
\begin{equation} \label{plane waves}
\w_\ell^{(1),\E}(\x):=e^{\im  \k_1  \dl^{\E} \cdot (\xbold - \xE)} \quad \forall \ell=1,\dots,\p^{\E},\quad \forall \x \in \E,
\end{equation}
the plane wave traveling along the directions~$\dl^{\E}$ with wave number~$\k_1$, we define the space of plane waves over~$\E$ as
\begin{equation} \label{bulk plane wave space}
\PW^{(1)}_{\p^{\E}}(\E) := \span \left\{ w_\ell^{(1),\E} \mid \ell= 1,\dots, \p^{\E} \right\}.
\end{equation}
Note that we allow here for elementwise different numbers of plane waves;
this notation is particularly suitable for developing the $\h\p$-version of the method, see Section~\ref{paragraph h VS hpISO}.
\medskip

Analogously, for all~$\E \in \taun^2$, we define the bulk plane wave space $\PW^{(2)}_{\p^{\E}}(\E)$ as the span of the plane waves~$\w_\ell^{(2),\E}$, which are defined in the same way as~$\w_\ell^{(1),\E}$ in~\eqref{plane waves}, but with wave number~$\k_2$ instead of~$\k_1$.

Following~\cite{luostari2013improvements, tezaur2008discontinuous}, we introduce a set of $\ptilde^{\E}=2\qtilde^{\E}$, $\qtilde^{\E}\in \mathbb N_{0}$, evanescent waves, for all~$\E \in \taun^2$. To this purpose, we first consider the set of equidistributed angles
\[
\thetaEW_{\elltilde} = \frac{\elltilde}{\qtilde^{\E}+1} \thetacrit \quad \forall \elltilde=1,\dots,\qtilde^{\E},
\]
where we recall that the critical angle~$\thetacrit$ is computed as in~\eqref{critical angle}. Then, the evanescent waves over~$\E$ are defined as
\begin{equation} \label{evanescent waves}
\wjEVE(\x):=e^{\im\k \widehat{\dir}^{\E}_{\frac{j}{2}} \cdot (\xbold - \xE)} \quad \forall j=1,\dots,\qtilde^{\E}, \quad \forall \x \in \E,
\end{equation}
where~$\k$ is the real number with $\k_1=n_1 \k$ and $\k_2=n_2 \k$, and $\widehat{\dir}^{\E}_{\frac{j}{2}} \in \R \times \C$ is given by
\begin{equation} \label{directions evanescent}
\widehat{\dir}^{\E}_{\frac{j}{2}}:=
\begin{cases}
\begin{pmatrix}
-n_1 \cos\left(\thetaEW_{\lceil \frac{j}{2} \rceil}\right),
\im \sqrt{n_1^2 \cos\left(\thetaEW_{\lceil \frac{j}{2} \rceil}\right)^2-n_2^2}
\end{pmatrix}
\quad &\text{if } j \text{ odd}\\
\begin{pmatrix}
n_1 \cos\left(\thetaEW_{\frac{j}{2}}\right),
\im \sqrt{n_1^2 \cos\left(\thetaEW_{\frac{j}{2}}\right)^2-n_2^2}
\end{pmatrix}
\quad &\text{if } j \text{ even}.
\end{cases}
\end{equation}

\begin{remark} \label{remark relaxing equidistribution}
Note that the assumption of having sets of equidistributed directions and angles in the construction of the plane and evanescent wave spaces, respetively, is made for the sake of simplicity and could be relaxed in principle,
without jeopardizing the approximation properties of the space of interest.
\end{remark}

As one can notice from~\eqref{evanescent waves} and~\eqref{directions evanescent}, the structure of an evanescent wave is similar to that of a plane wave; the difference is that the direction vector is complex-valued in the former case, whereas it is real-valued in the latter.
As discussed and numerically proven in~\cite{luostari2013improvements, tezaur2008discontinuous}, the evanescent waves are better suited than plane waves to capture the exponential decay of the evanescent modes appearing in the fluid-fluid interface problem
for specific incident angles~$\thetainc$, and therefore they could be added to the approximation space associated with the domain~$\Omega_2$ to improve the performance of the method.

We point out that the evanescent waves given by \eqref{evanescent waves} satisfy the homogeneous Helmholtz problem in~$\Omega_2$.
In Figure~\ref{fig:ew1_ew2}, we plot the real and imaginary part of the evanescent wave with parameters~$k=5$, $n_1=2$ and~$n_2=1$ (critical angle $\thetacrit=60^\circ$), and~$\xE=(0,0)$.

\begin{figure}[h]
\begin{center}
\begin{minipage}{0.48\textwidth} 
\includegraphics[width=\textwidth]{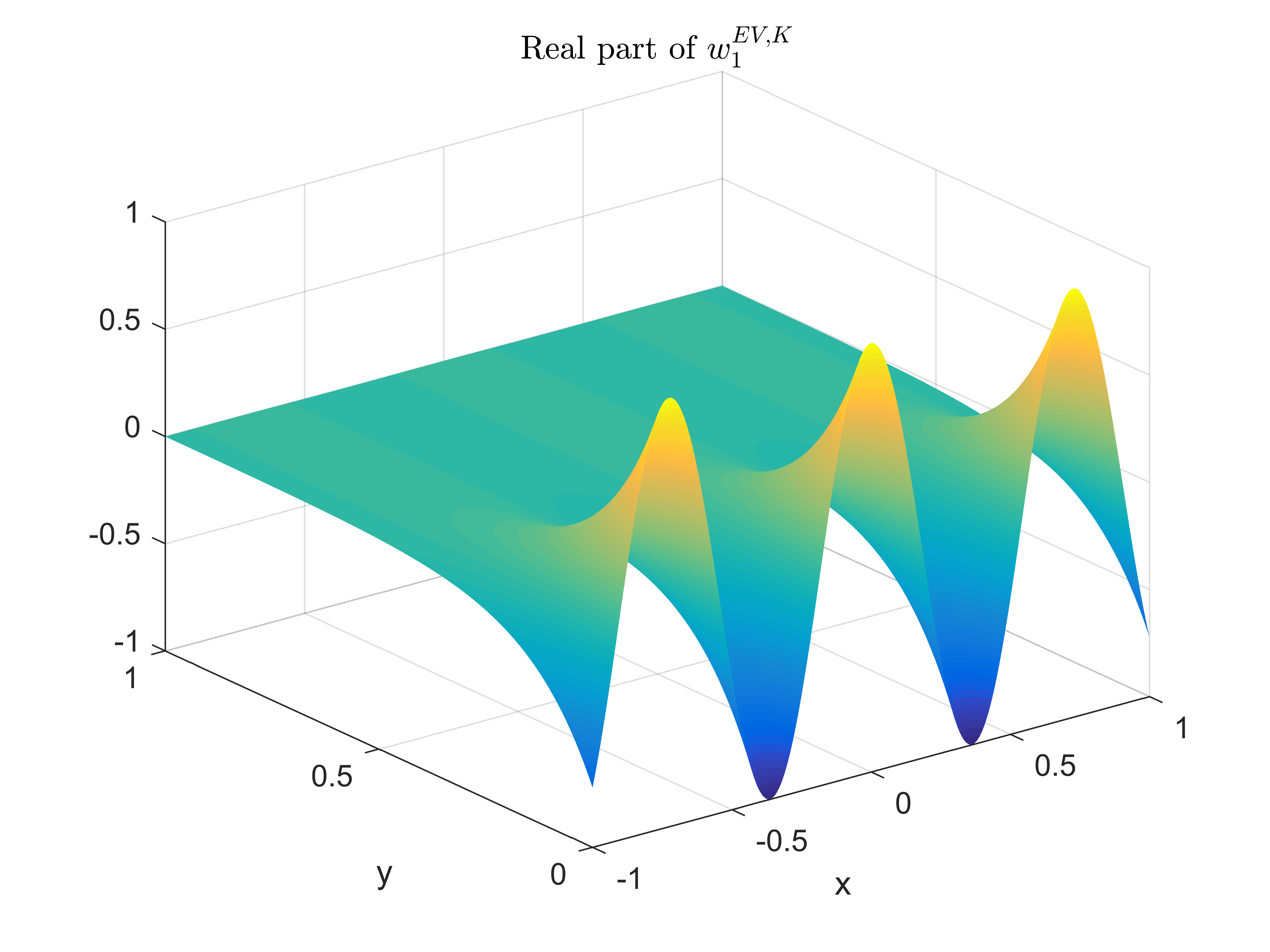}
\end{minipage}
\hfill
\begin{minipage}{0.48\textwidth}
\includegraphics[width=\textwidth]{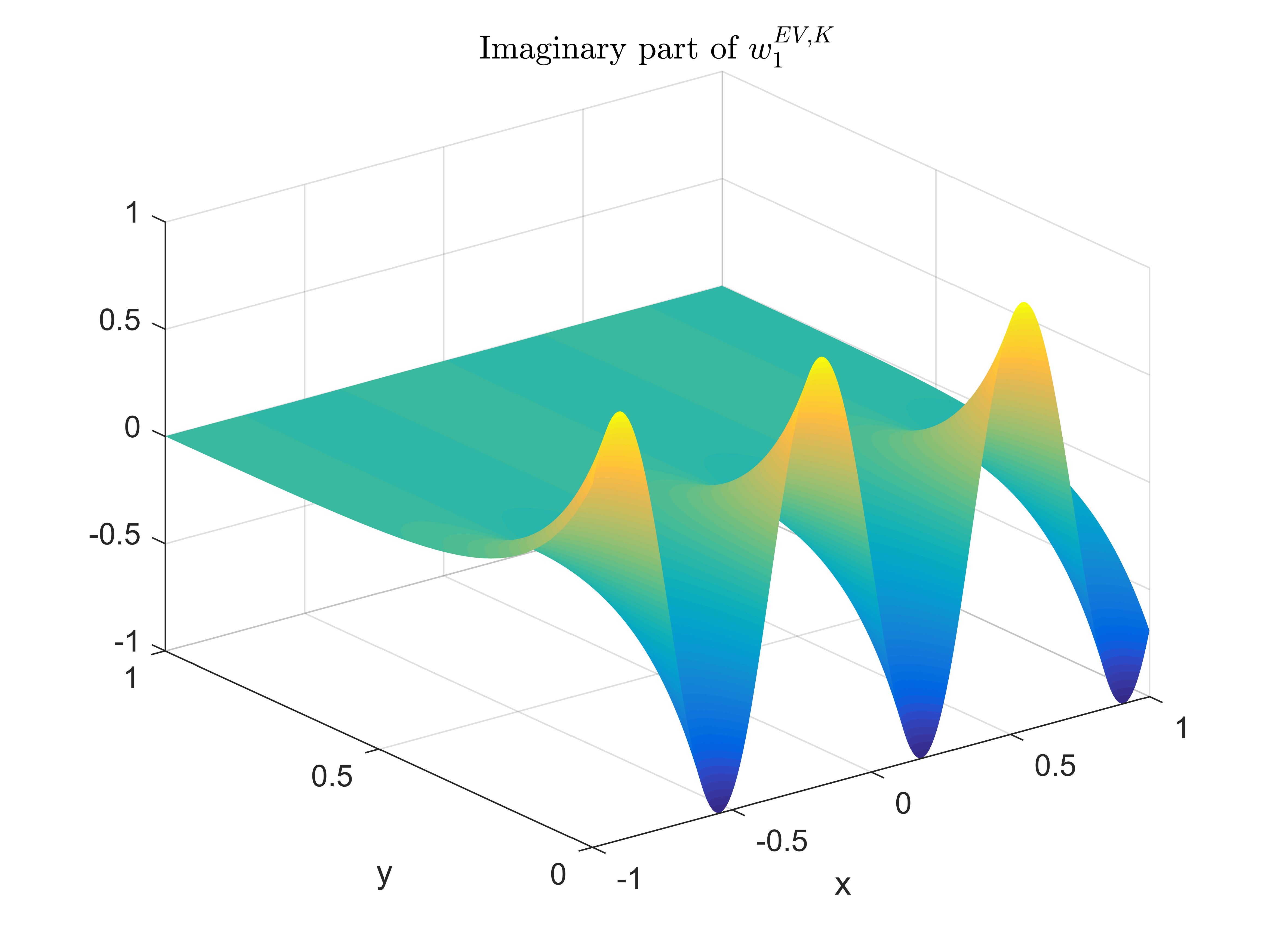}
\end{minipage}
\end{center}
\caption{Real and imaginary parts of the first evanescent wave for~$k=5$, $n_1=2$, $n_2=1$, and~$\xE=(0,0)$.}
\label{fig:ew1_ew2} 
\end{figure}

Finally, we define the space of \emph{evanescent waves} over~$\E\in \taun^2$
\begin{equation*}
\EW_{\ptilde^{\E}}(\E) := \span \left\{ \wjEVE  \mid j=1,\dots, \ptilde^{\E} \right\},
\end{equation*}
and the space of plane waves \emph{and} evanescent waves
\begin{equation} \label{bulk plane and evanescent waves}
\PWtildepqtilde (\E) := \PW^{(2)}_{\p^{\E}}(\E) \cup \EW_{\ptilde^{\E}}(\E).
\end{equation}
In the following, we shall also need spaces of traces of plane waves and evanescent waves over edges. For all edges~$\e \in \En$, we set
\begin{equation} \label{edge plane-evanescent waves}
\PWtilde(\e) := 
\begin{cases}
{\PW^{(1)}_{\p^{\E}}(\E)}_{|\e}, & \text{if } e \in \EnBone \cap \EE \\
{\PW^{(1)}_{\p^{\E^-}}(\E^-)}_{|\e} \cup {\PW^{(1)}_{\p^{\E^+}}(\E^+)}_{|\e} , &
\begin{aligned}
\text{if } & \e \in \EnIone \cap \mathcal{E}^{\E^-} \cap \mathcal{E}^{\E^+} \text{ with } \E^- \neq \E^+ \\ & \E^+,\, \E^- \in \taun^1\\
\end{aligned}\\
{\PWtildepqtilde(\E)}_{|\e} & \text{if } e \in \EnBtwo \cap \EE \\
\PWtildepqtildeminus(\E^-)_{|\e} \cup \PWtildepqtildeplus(\E^+)_{|\e} &
\begin{aligned}
\text{if } 	& \e \in \EnItwo \cap \mathcal{E}^{\E^-} \cap \mathcal{E}^{\E^+} \text{ with } \E^- \neq \E^+ \\
		& \E^+,\, \E^- \in \taun^2\\
\end{aligned}\\
{\PW_{\p^{\E^-}}^{(1)}(\E^-)}_{|\e} \cup \PWtildepqtildeplus(\E^+)_{|\e}, &  
\begin{aligned}
\text{if } &\e \in \EnGamma \cap \mathcal{E}^{\E^-} \cap \mathcal{E}^{\E^+} \text{ with } 
\\ &\quad\quad \E^- \in \taun^1, \E^+ \in \taun^2,
\end{aligned}
\end{cases}
\end{equation}
denoting by~$\pe$ the dimension of the space~$\PWtilde(\e)$.

In words, we consider spaces of traces of plane waves with wave number~$\k_1$ on all edges in $\EnIone \cup \EnBone$, spaces of traces of plane waves with wave number~$\k_2$ and evanescent waves on all edges in $\EnItwo \cup \EnBtwo$,
and, at the interface~$\Gamma$, we consider traces of plane waves with the two different wave numbers~$k_1$ and~$k_2$ and evanescent waves.
The definition~\eqref{edge plane-evanescent waves} will be instrumental to build suitable nonconforming Sobolev spaces.

\begin{remark} \label{remark orthogonalization}
Whilst the dimensions of the bulk plane wave spaces $\PW^{(1)}_{\p^{\E}}(\E)$ and~$\PWtildepqtilde (\E)$ are given by~$\p^{\E}$ and~$\p^{\E}+\ptilde^{\E}$, respectively,
those of the spaces~$\PWtilde(\e)$ are in general smaller than or equal to the sum of the dimensions of the bulk spaces of the adjacent polygons.
In fact, the restriction of two different plane waves onto a given edge could generate a 1D space only.
On the other hand, whenever the restrictions of two plane waves with different directions and wave numbers on a given edge are ``close'', numerical instabilities may occur.
In order to avoid this situation, we will employ the edgewise or\-tho\-go\-na\-li\-za\-tion-and-fil\-te\-ring process introduced in \cite{TVEM_Helmholtz_num}, see Section~\ref{subsection implementational aspects}. 
\end{remark}

Next, we define the broken Sobolev space of order~$s\in \mathbb N$, subordinated to a polygonal decomposition~$\taun$:
\[
H^s(\taun) := \left\{v \in L^2(\Omega) \mid v_{|\E} \in H^s(\E) \; \forall \E \in \taun \right\},
\]
with the seminorms and weighted norms
\[
\vert v \vert^2 _{s,\taun} := \sum_{\E \in \taun} \vert v \vert ^2_{s,\E};
\quad \Vert v \Vert^2_{s,\k; \taun} := \sum_{\E \in \taun} \Vert v \Vert^2_{s,\k,\E} = \sum_{\E \in \taun} \sum_{j=0}^s \k^{2(s-j)} \vert v \vert^2_{j,\E}.
\]
In order to introduce the global nonconforming Sobolev space, we need some additional notation.
Given~$\e \in \EnI$ with adjacent elements~$\E^+$ and~$\E^-$, we set~$\n_{\E^{\pm}}$ the two outer unit normal vectors with respect to~$\partial \E^\pm$.
Further, we define the vector-valued jump of $v \in H^1(\taun)$ across the edge~$\e$ as
\[
\llbracket v \rrbracket_\e := v_{|{\E^+}} \n_{\E^+} + v_{|{\E^-}} \n_{\E^-}.
\]
We will use the notation~$\llbracket v \rrbracket$ instead of$~\llbracket v \rrbracket_\e$ when no confusion occurs.

The global nonconforming Sobolev space with edgewise order of nonconformity~$\pe$ is built as follows.
Given~$N_j$ the cardinality of $\taun^j$, $j=1,2$, we consider the vectors~$\pboldo \in [\mathbb N_{\ge 3}]^{N_1}$, $\pboldt \in [\mathbb N_{\ge 3}]^{N_2}$, and $\pboldtilde \in [\mathbb N_0]^{N_2}$,
representing the distribution of the dimensions of the bulk plane wave spaces over the elements in~$\taun^1$, and of the bulk plane wave spaces and of the evanescent wave spaces over the elements in~$\taun^2$, respectively.
To the set of edges~$\En$, we associate a vector~$\pboldepsilon \in \mathbb N ^{\card(\En)}$, whose $j$-th entry represents the dimension of the space~$\PWtilde(\e)$ defined in~\eqref{edge plane-evanescent waves} on the $j$-th global edge~$\e$.

Eventually, we define the global nonconforming Sobolev space associated with the vector~$\pboldepsilon$:
\begin{equation} \label{nonconforming space}
H_{\pboldepsilon}^{1,nc}(\taun) := \left\{v \in H^1(\taun)  \mid \int_e \llbracket v \rrbracket \cdot \n^\e \overline {\we} \, \ds=0 \quad\forall \we \in \PWtilde(\e),\, \forall \e \in \EnI \right\}.
\end{equation}
We highlight that by using this construction, nonconforming Sobolev spaces can be straightforwardly generalized to the case of piecewise constant~$\kcal$ on more than two subdomains.

\section{A nonconforming Trefftz virtual element method for the fluid-fluid interface problem} \label{section ncTVEM fluidfluid}
In this section, we introduce a nonconforming Trefftz-VEM for the approximation of the fluid-fluid interface problem~\eqref{weak monolithic formulation} based on plane waves and evanescent waves.
Such a method differs from the original one in~\cite{ncTVEM_Helmholtz,TVEM_Helmholtz_num} by the two following features:
\begin{itemize}
\item the wave number is piecewise (and not globally) constant;
\item special functions, i.e., evanescent waves, are locally added to the approximation spaces to capture the physical behaviour of the evanescent modes possibly appearing in~$\Omega_2$ in proximity of the interface~$\Gamma$. 
\end{itemize}
We will see that these two features elegantly fit within the nonconforming VEM setting of~\cite{ncHVEM, ncTVEM_Helmholtz, TVEM_Helmholtz_num}.

The method we design has the following structure:
\begin{equation} \label{VEM}
\begin{cases}
\text{find } \uh \in \Vh \text{ such that}\\
\bh(\uh,\vh) =\Fh(\vh) \quad \forall \vh \in \Vh,
\end{cases}
\end{equation}
where~$\Vh$ is a finite dimensional space, $\bh(\cdot,\cdot): [\Vh]^2\rightarrow \mathbb C$ is a computable sesquilinear form mimicking its continuous counterpart~$\b(\cdot,\cdot)$ defined in~\eqref{complete form}, and
$\Fh(\cdot):\Vh \rightarrow \mathbb C$ is a computable functional mimicking its continuous counterpart~$\langle \g,\cdot \rangle$ in~\eqref{rhs}.

The remainder of the section is organized as follows. In Section~\ref{subsection VE spaces}, we introduce the local and global nonconforming Trefftz virtual element spaces,
together with a set of unisolvent degrees of freedom.
Next, in Section~\ref{subsection local operators}, we introduce a couple of local (bulk and edge) projectors from local virtual element spaces into proper (plane/evanescent) wave spaces.
Such operators, in addition to proper suitable stabilizations, are instrumental for the construction of the discrete sesquilinear form~$\bh(\cdot,\cdot)$ and right-hand side~$\Fh(\cdot)$ in~\eqref{VEM}, which is the topic of Section~\ref{subsection forms and rhs}.

Henceforth, we will assume that three distributions~$\pboldo$, $\pboldt$, and~$\pboldtilde$, as in the construction of the nonconforming Sobolev spaces in~\eqref{nonconforming space}, are given, and that~$\pboldepsilon$ is the resulting edge distribution.

\subsection{Local Trefftz virtual element spaces and global nonconforming spaces} \label{subsection VE spaces}
Our aim here is to introduce local Trefftz-VE spaces tailored for the fluid-fluid interface problem~\eqref{weak monolithic formulation}, and subsequently to patch them into a global space in a nonconforming fashion.

To this purpose, given~$\E \in \taun$, we set the local space
\begin{equation} \label{local spaces}
\VhE := \{\vh \in H^1(\E) \mid \Delta \vh + \kcal^2 \vh = 0 \text{ in }\E,\; (\nabla \vh \cdot \nE + \im \kcal \vh) _{|\e} \in \PWtilde(\e)\, \forall \e \in \EE\},
\end{equation}
where we recall that the edge spaces~$\PWtilde(\e)$ are defined in~\eqref{edge plane-evanescent waves}.

We point out that, for every element~$\E \in \taun^1$, the space~$\VhE$ contains~$\PW^{(1)}_{\p^\E}(\E)$, the space of~$\pE=2\qE+1$ plane waves with wave number~$\k_1$  defined in~\eqref{bulk plane wave space};
besides, it contains additional functions that are not known in closed form (whence the name \textit{virtual}) and that are locally Trefftz with impedance traces in the space~$\PWtilde(\e)$, for all edges~$\e \in \EE$.

On the other hand, the local spaces over the elements~$\E \in \taun^2$ are designed in such a way that they contain~$\PWtildepqtilde(\E)$,
the space of~$\pE=2\qE+1$ plane waves  with wave number~$\k_2$ and~$\ptildeE=2\widetilde\q_\E$ evanescent waves defined in~\eqref{bulk plane and evanescent waves};
again, there are additional functions unknown in closed form inside (which however have impedance traces in the space of traces of plane and evanescent waves).
Such additional functions will be instrumental for building nonconforming global spaces, as described below.

Henceforth, we call~$\qE$ the \emph{effective degree} of the method on the elements~$\E\in \taun^1$, and~$\qE+\widetilde\q _\E$ the \emph{effective degree} of the method on the elements~$\E\in \taun^2$.


Given~$\E \in \taun$ and the associated local Trefftz-VE space~$\VhE$, we consider the following set of linear functionals on~$\VhE$. For all~$\e\in \EE$,
\begin{equation} \label{dofs}
\dof_{\e,\alpha} (\vh) := \frac{1}{\he} \int_\e \vh \overline{w_\alpha^\e} \, \ds \quad \forall \alpha=1,\dots,\pe,
\end{equation}
where $\{w_\alpha^\e\} _{\alpha=1} ^{\pe}$ is \emph{any} basis for the space~$\PWtilde(\e)$. This set of functionals is a set of unisolvent degrees of freedom, as stated in the following result.
\begin{lem} \label{lemma unisolvency}
Given~$\E \in \taun$, let us assume that~$\kcal_{|\E}$ is not a Dirichlet-Laplace eigenvalue on~$\E$. Then, the set of functionals defined in~\eqref{dofs} is a unisolvent set of degrees of freedom
for the space~$\VhE$.
\end{lem}
\begin{proof}
The proof follows the lines of that of \cite[Lemma 3.1]{ncTVEM_Helmholtz} and is therefore omitted here.
\end{proof}

\begin{remark}
Note that the assumption on~$\kcal_{|\E}$ in Lemma \ref{lemma unisolvency} actually results in a condition on the size of the product~$\hE \kcal_{|\E}$, see~\cite{ncTVEM_Helmholtz}.
More precisely, for~$\hE$ sufficiently small, $\kcal_{|\E}$ is not a Dirichlet-Laplace eigenvalue on~$\E$.
\end{remark}

Having this, we introduce the set of local canonical basis functions $\{\varphi_{\hat{e},\hat{\alpha}}\}_{\hat{e},\hat{\alpha}}$ by duality:
\begin{equation*}
\dof_{\e,\alpha} (\varphi_{\hat{e},\hat{\alpha}}) := \delta_{e,\hat{e}} \delta_{\alpha,\hat{\alpha}}, \quad \forall \e,\hat{\e} \in \En, \, \forall \alpha = 1,\dots,\pe, \, \forall \hat{\alpha}=1,\dots,\p_{\hat \e},
\end{equation*}
where~$\delta_{\cdot, \cdot}$ here denotes the standard Kronecker delta.

The choice of the degrees of freedom in~\eqref{dofs} together with the definition of the spaces~$\PWtilde(\e)$ in~\eqref{edge plane-evanescent waves} allows for the construction of the global nonconforming Trefftz virtual element space
\begin{equation} \label{global nc spaces}
\Vh:= \{\vh \in H_{\pboldepsilon}^{1,nc} (\taun) \mid \vh{}_{|\E} \in \VhE \, \forall \E \in \taun \},
\end{equation}
where we recall that the nonconforming Sobolev space~$H_{\pboldepsilon}^{1,nc}(\taun)$ is defined in~\eqref{nonconforming space}.

Moreover, the global set of the degrees of freedom is built by a nonconforming coupling (\emph{\`{a} la} Crouzeix-Raviart) of the local counterparts~\eqref{dofs}, see~\cite{ncTVEM_Helmholtz}.


\subsection{Local projectors} \label{subsection local operators}
In this section, we introduce a couple of local projectors which will be instrumental for the design of the method~\eqref{VEM}.

First of all, for all~$\E \in \taun^1$, we define the local operator $\PiEp: \VhE \rightarrow \PW^{(1)}_{\p^\E}(\E)$ by 
\begin{equation} \label{Pi projector}
\aE(\PiEp\vh, w^{(1),\E}) = \aE(\vh, w^{(1),\E}) \quad \forall \vh \in \VhE,\; \forall w^{(1),\E} \in \PW^{(1)}_{\p^\E}(\E).
\end{equation}
Such operator is computable by means of the degrees of freedom~\eqref{dofs}. In fact, an integration by parts yields
\begin{equation*}
\aE(\vh, w^{(1),\E}) = \int_{\partial \E} \vh \overline{\nabla w^{(1),\E}  \!\! \cdot \nE} \, \ds,
\end{equation*}
which is computable since $(\nabla w^{(1),\E} \!\! \cdot \nE)_{|\e} \in \PWtilde(\e)$ for all~$\e \in \EE$.

Besides, $\PiEp$ is well-defined under the assumption that the size of the element~$\E$ is sufficiently small, see~\cite[Proposition 3.2]{ncTVEM_Helmholtz} for more details.
\medskip

For all~$\E \in \taun^2$, we also introduce the local projector $\PiEpqtilde : \VhE \rightarrow \PWtildepqtilde (\E)$ which is defined analogously to~$\PiEp$ in~\eqref{Pi projector} with the only difference that the space $\PW^{(1)}_{\p^\E}(\E)$ is replaced by $\PWtildepqtilde (\E)$. 
The well-posedness of~$\PiEpqtilde$ is provided by the invertibility of the matrix $\GboldEtwohat \in \mathbb C^{(\phEhat + \ptildeE) \times (\phEhat + \ptildeE)}$ defined by
\begin{equation} \label{matrix Gbold}
\GboldEtwohat_{j,\ell} := 
\begin{cases}
\aE(w_\ell^{(2),{\E}}, w_j^{(2),{\E}}), & \text{ if } j,\ell \leqslant \phEhat \\

\aE(w_{\ell-\phEhat}^{EV,{\E}}, w_j^{(2),{\E}}), & \text{ if } j \leqslant \phEhat, \ell > \phEhat \\
\aE(w_\ell^{(2),{\E}}, w_{j-\phEhat}^{EV,{\E}}), & \text{ if } j > \phEhat, \ell \leqslant \phEhat \\
\aE(w_{\ell-\phEhat}^{EV,{\E}}, w_{j-\phEhat}^{EV,{\E}}), & \text{ if } j,\ell > \phEhat \\
\end{cases}
\end{equation}
for all $j,\ell=1,\dots,\phEhat + \ptildeE$.

By investigating the behaviour of the minimal (absolute) eigenvalue of~$\GboldEtwohat$ in terms of the wave number~$k_2$ on the reference element~${\E}:=(0,1)^2$,
one can observe that such a minimal eigenvalue becomes very small when~$k_2^2$ is close to a Neumann-Laplace eigenvalue $\nu_{m,n}:=\pi^2(m^2+n^2)$, $m,n \in \N_{0}$, on ${\E}$, see Figure~\ref{fig:mineig}.
\begin{figure}[h]
\centering
\includegraphics[width=0.5\linewidth]{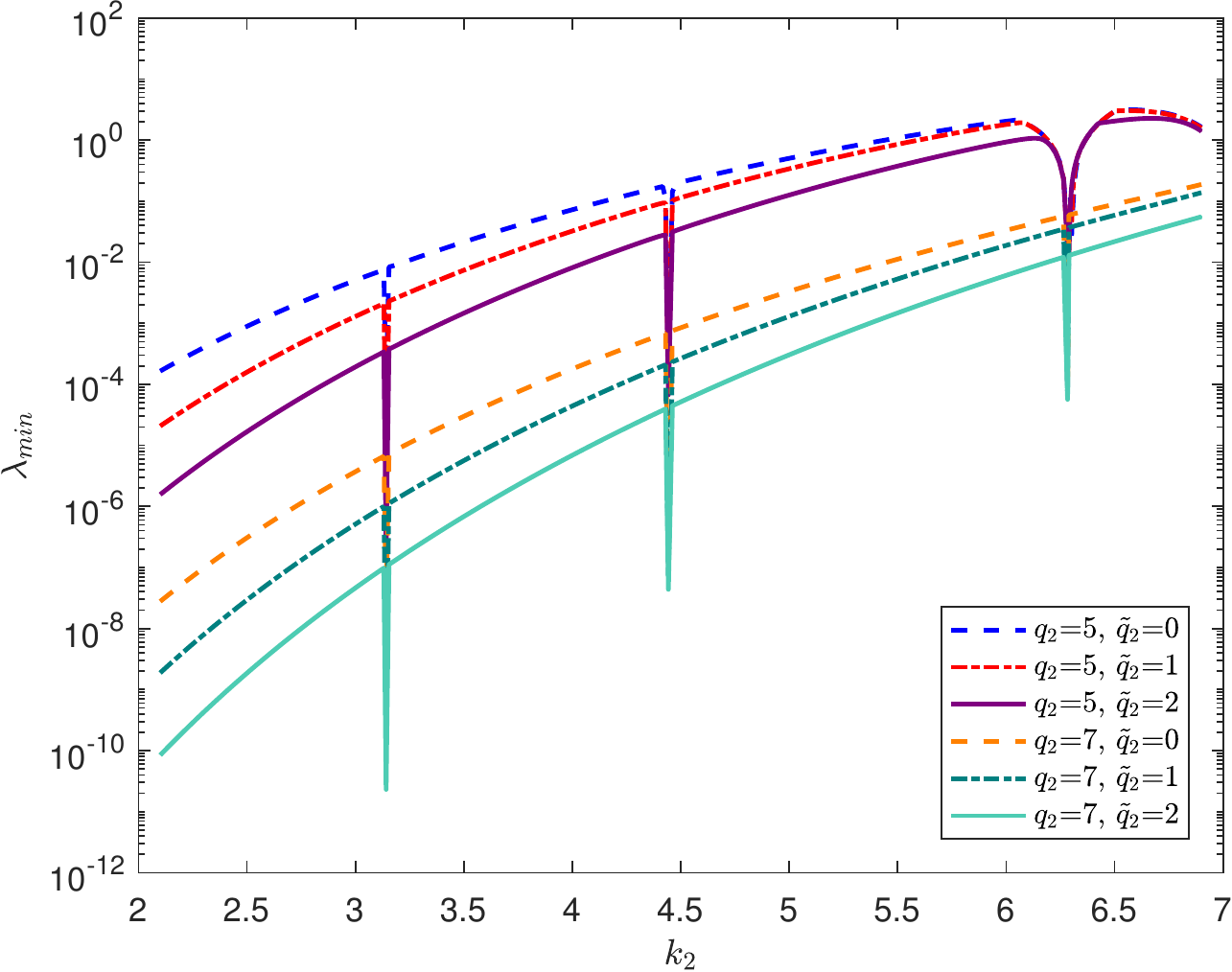}
\caption{Minimal (absolute) eigenvalues of the matrix~$\GboldEtwohat$ in~\eqref{matrix Gbold} in terms of the wave number~$k_2$ with~$n_1=2$ and~$n_2=1$.
The effective plane and evanescent wave degrees are denoted by~$q_2$ and~$\widetilde q_2$, respectively.}
\label{fig:mineig}
\end{figure}
This indicates that, assuming~$\k_2^2$ to be separated from the Neumann-Laplace eigenvalues, the local projector~$\PiEpqtilde$ is well-defined.
\medskip

The third operator we introduce is the boundary edge~$L^2$ projector $\Pizepqtilde: \VhE{}_{|\e} \rightarrow \PWtilde(\e)$, which is defined for all edges~$\e \in \Enb$ by
\begin{equation*}
(\Pizepqtilde \vh, w^e)_{0,\e} = (\vh,w^e)_{0,\e} \quad \forall \vh \in \VhE{}_{|\e},\;\forall w^e \in \PWtilde(\e).
\end{equation*}
Such a projector is directly computable starting from the local degrees of freedom in~\eqref{dofs};
moreover, it is well-defined owing to the coercivity of the edge~$L^2$ norm.

\subsection{Discrete sesquilinear forms and right-hand side} \label{subsection forms and rhs}
Here, we specify the discrete sesquilinear form~$\bh(\cdot,\cdot)$ and the discrete right-hand side~$\Fh(\cdot)$ characterizing the method~\eqref{VEM}.

To begin with, we underline that the continuous counterparts~$\b(\cdot,\cdot)$ and~$\langle \g,\cdot \rangle$ in~\eqref{complete form} and~\eqref{rhs}, respectively, are in general not explicitly computable when applied to functions in~$\Vh$ defined in~\eqref{global nc spaces},
since the functions in the global nonconforming Trefftz-VE space are not known in closed form.

Therefore, we proceed following the standard VEM gospel~\cite{VEMvolley}. First, we introduce, for all~$\E \in \taun$, local stabilizing sesquilinear forms $\SE(\cdot,\cdot):[\ker(\PiE)]^2 \rightarrow \mathbb C$,
where~$\PiE$ is either~$\PiEp$ or~$\PiEpqtilde$, depending on whether~$\E \in \taun^1$ or~$\E \in \taun^2$; such sesquilinear forms are referred to as \emph{stabilizations} and they
have to be computable  employing only the degrees of freedom of the local space~$\VhE$, see~\eqref{dofs}.

Depending on the choice of the stabilizations~$\SE(\cdot,\cdot)$, we propose a family of discrete sesquilinear forms~$\bh(\cdot,\cdot)$ characterizing method~\eqref{VEM}:
\[
\bh(\uh,\vh) := \sum_{\E \in \taun} \ahE(\uh{}_{|\E} , \vh{}_{|\E}) + \im \ch (\kcal  \uh, \vh) \quad \forall \uh,\,\vh \in \Vh,
\]
where, for all~$\E\in \taun$,
\begin{equation} \label{local discrete sesquilinear form}
\ahE(\uh,\vh) := \aE(\PiE \uh, \PiE \vh)+ \SE( (I - \PiE ) \uh, (I-\PiE ) \vh) \quad \forall \uh,\, \vh \in \VhE,
\end{equation}
with~$\PiE=\PiEp$ for all~$\E \in \taun^1$ and $\PiE=\PiEpqtilde$ for all~$\E \in \taun^2$, and where
\begin{equation*}
\ch (\kcal  \uh, \vh) := \sum_{\e \in \Enb} (\kcal \Pizepqtilde ({\uh}_{|\e}), \Pizepqtilde ({\vh}_{|\e}) )_{0,\e} \quad \forall \uh,\, \vh \in \Vh.
\end{equation*}
A discussion on the requirements that the stabilizations~$\SE(\cdot,\cdot)$ have to satisfy in order to entail well-posedness and error estimates of the method~\eqref{VEM} is the object of~\cite[Theorem 4.4]{ncTVEM_Helmholtz}.
An explicit choice for the stabilization~$\SE(\cdot, \cdot)$ is provided in~\eqref{modified D-recipe}.
\medskip

The discrete right-hand side is defined as
\begin{equation*}
\Fh(\vh) := \sum_{\e \in \Enb} (\g, \Pizepqtilde (\vh{}_{|\e}))_{0,\e} \quad \forall \vh \in \Vh.
\end{equation*}
Note that the right-hand side is approximated by employing 1D quadrature formulas. In fact, this is the only occurrence where quadrature formulas are needed.

\section{Details on the implementation and numerical results} \label{section numerical results}
In this section, we first discuss some details of the implementation of method~\eqref{VEM} in Section~\ref{subsection implementational aspects}, and then, we present numerical experiments for a series of different test cases in Section~\ref{subsection numerical experiments}.

\subsection{Implementation aspects} \label{subsection implementational aspects}
The implementation of the method is performed analogously to the case of constant~$\kcal$, see~\cite{TVEM_Helmholtz_num}. However, for the sake of clarity and completeness, we will give a few details below.
It is of great importance to underline that the implementation of the method follows the lines of that of standard nonconforming FEM (and VEM);
in particular, local matrices are computed and eventually patched into a global one.

\paragraph*{Or\-tho\-go\-na\-li\-za\-tion-and-fil\-te\-ring process.}
First of all, we highlight that (cf. Remark~\ref{remark orthogonalization}) for all edges~$\e\in \En$
we will not directly use the traces of plane waves and evanescent waves defined in~\eqref{plane waves} and~\eqref{evanescent waves}, respectively, as basis functions for the spaces~$\PWtilde(\e)$.
In fact, by doing that, we would bump into numerical instabilities due to the high condition numbers of the local~$L^2$ edge mass matrices related to these basis functions, see~\cite{TVEM_Helmholtz_num}.
Instead, we will use the numerical recipe based on an or\-tho\-go\-na\-li\-za\-tion-and-fil\-te\-ring process proposed in \cite[Algorithm 2]{TVEM_Helmholtz_num},
which allows $(i)$ to automatically filter out redundancies in the edge basis functions, depending on the choice of a filtering parameter, and $(ii)$ to reduce the number of degrees of freedom needed for the convergence of the method, as discussed in~\cite[Section 5.3]{TVEM_Helmholtz_num};
see Algorithm~\ref{algorithm orthog process}.
\begin{algorithm}[h]
\caption{}
\label{algorithm orthog process}
Let~$\sigma>0$ be a given ``filtering'' tolerance. For all the edges~$e \in \En$:
\begin{enumerate}
\item Assemble the real-valued, symmetric, and possibly singular matrix $\boldsymbol{G}_0^e \in \R^{\rho_\e \times \rho_\e}$ given by
\begin{equation*}
(\boldsymbol{G}_0^e)_{j,\ell}= (\nu_{\ell}^e,\nu_j^e)_{0,e}
\quad \forall j,\ell=1,\dots,\rho_\e,
\end{equation*}
where~$\nu_{\ell}^e$ are the traces of \textit{all} the basis functions belonging to the edge space~$\PWtilde(\e)$ defined in~\eqref{edge plane-evanescent waves}. Let~$\rho_\e$ be their number. 
\item Compute the eigendecomposition:
\begin{equation*}
\boldsymbol {G}_0^\e \boldsymbol{Q}^\e =   \boldsymbol{Q}^\e \boldsymbol {\Lambda}^\e,
\end{equation*}
where $\boldsymbol{Q}^e \in \R^{\rho_\e \times \rho_\e}$ is a matrix whose columns are right-eigenvectors, and $\boldsymbol{\Lambda}^\e \in\R^{\rho_\e \times \rho_\e}$ is a diagonal matrix containing the corresponding eigenvalues. 
\item Determine the eigenvalues with (absolute) value smaller than the tolerance~$\sigma$ and remove the columns of~$\boldsymbol{Q}^e$ corresponding to these eigenvalues.
Denote the number of remaining columns of~$\boldsymbol{Q}^e$ by~$\pehat \le \rho_\e$.
The remaining columns of~$\boldsymbol{Q}^e$ are relabeled by~$1,\dots,\pehat$.
\item Define the new~$L^2(\e)$ orthogonal edge functions~$\wlehat$, $\ell=1,\dots,\pehat$, in terms of the old ones~$\nu_r^e$, $r=1,\dots,\rho_\e$, as
\begin{equation*}
\wlehat:=\sum_{r=1}^{\rho_e} \boldsymbol{Q}^e_{r,\ell} \, \nu_r^e.
\end{equation*}
\end{enumerate}
\end{algorithm}

Importantly, the above-mentioned strategy naturally dovetails with the supplement of special functions to the standard plane wave spaces and the use of plane wave spaces with varying degree from element to element.
The traces of the corresponding functions are simply added edgewise first, as they are needed for the construction of the method (this leads to an increase of the number of degrees of freedom);
afterwards, the relevant information is extracted using Algorithm~\ref{algorithm orthog process} and the number of degrees of freedom is reduced significantly.
In all the forthcoming numerical experiments, the tolerance~$\sigma$ will be set to~$10^{-13}$. The effect of the choice of~$\sigma$ on the performance of the method was investigated in~\cite{TVEM_Helmholtz_num}, in the case of constant~$\kcal$.

Henceforth, we use the convention that the local degrees of freedom and canonical basis functions associated to the orthogonalized basis functions~$\wlehat$ will be hooded by a hat.

\paragraph*{Global and local matrices.}
As usual in the standard nonconforming FEM and VEM philosophy, the global system of linear equations is assembled in terms of the local elementwise contributions.
Setting $\phatE:=\sum_{\e \in \EE} \pehat$ and recalling that~$\Ne$ denotes the number of edges of~$\E$, we define the following matrices, see~\cite{hitchhikersguideVEM, TVEM_Helmholtz_num}:
\begin{itemize}
\item for all~$\E \in \taun^1$:
\begin{itemize}
\item[*] $\GboldEone \in \mathbb C^{\phE \times \phE}$ with
$\GboldEone_{j,\ell} := \aE(w_\ell^{(1),\E}, w_j^{(1),\E})$, for all $j,\ell=1,\dots,\phE$;
\item[*] $\DboldEone \in \mathbb C^{\phatE \times \phE}$ with $\DboldEone_{(r,j),\ell} := \widehat{\dof}_{r,j}(w_\ell^{(1),\E})$, for all $r=1,\dots,\Ne$, $j=1,\dots,\widehat{p}_{e_r}$, and $\ell=1,\dots,\phE$;
\item[*] $\BboldEone \in \mathbb C^{\phE \times \phatE}$ with $\BboldEone_{j,(s,\ell)} := \aE(\widehat{\varphi}_{s,\ell},w_j^{(1),\E})$, for all $j=1,\dots,\phE$, $s=1,\dots,\Ne$, and $\ell=1,\dots,\widehat{p}_{e_s}$;
\end{itemize} 
\item for all $\E \in \taun^2$:
\begin{itemize}
\item[*] $\GboldEtwo \in \mathbb C^{(\phE + \ptildeE) \times (\phE + \ptildeE)}$ as in~\eqref{matrix Gbold};
\item[*] $\DboldEtwo \in \mathbb C^{\phatE \times (\phE+\ptildeE)}$ with 
\begin{equation*}
\DboldEtwo_{(r,j),\ell} :=
\begin{cases} 
\widehat{\dof}_{r,j}(w_\ell^{(2),\E}),  & \text{ if } \ell \le \phE \\
\widehat{\dof}_{r,j}(w_{\ell-\phE}^{EV,\E}),  & \text{ if } \ell > \phE, \\
\end{cases}
\end{equation*}
for all $r=1,\dots,\Ne$ and $j=1,\dots,\widehat{p}_{e_r}$;
\item[*] $\BboldEtwo \in \mathbb C^{(\phE+\ptildeE) \times \phatE}$ with 
\begin{equation*}
\BboldEtwo_{j,(s,\ell)} :=
\begin{cases}
\aE(\widehat{\varphi}_{s,\ell},w_j^{(2),\E}), & \text{ if } j \le \phE \\
\aE(\widehat{\varphi}_{s,\ell},w_{j-\phE}^{EV,\E}), & \text{ if } j > \phE,
\end{cases}
\end{equation*}
for all $s=1,\dots,\Ne$, and $\ell=1,\dots,\widehat{p}_{e_s}$;
\end{itemize} 
\end{itemize}
Having this, following \cite{TVEM_Helmholtz_num}, the matrix representation $\AboldEone$ of $\ahE(\cdot,\cdot)$ is given, for all $\E \in \taun^1$, by
\[
\AboldEone := \overline{\PiboldEsone}^T \GboldEone \PiboldEsone + (\overline{\IboldEone-\PiboldEone})^T \SboldEone ( \IboldEone - \PiboldEone),
\]
where $\IboldEone \in \mathbb C^{\phatE \times \phatE}$ is the identity matrix, $\SboldEone$ is the matrix representation of the stabilizing sesquilinear form $\SE(\cdot,\cdot)$, and 
\[
\PiboldEsone := (\GboldEone)^{-1} \BboldEone \in \mathbb C^{\phE \times \phatE},\quad 
\PiboldEone: = \DboldEone (\GboldEone)^{-1} \BboldEone \in \mathbb C^{\phatE \times \phatE}.
\]
The matrix~$\AboldEtwo$ related to~$\ahE(\cdot,\cdot)$ for~$\E \in \taun^2$ is computed analogously.

Regarding the Robin part, given~$\e \in \Enb$, the local matrix representation~$\Rbolde$ of $(\kcal \Pizepqtilde \cdot, \Pizepqtilde \cdot)_{0,\e}$ is
\[
\Rbolde:= \overline{\Bbolde}^T \overline{\Gbolde}^{-T} \Bbolde,
\]
where~$\Gbolde$ and $\Bbolde \in \C^{\pehat \times \pehat}$ are given by $(\Gbolde)_{j,\ell}:=(\widehat{w}_\ell^e,\widehat{w}_j^e)_{0,e}$ and  $(\Bbolde)_{j,\ell}:=(\widehat{\varphi}_{\e,\ell},\widehat{w}_j^e)_{0,e}=\he \delta_{j,\ell}$, for all $j,\ell=1,\dots,\pehat$, respectively.

The right-hand side $\Fh(\vh)$ is computed by expressing $\Pizepqtilde (\vh{}_{|\e})$ in terms of the orthogonalized basis functions $\widehat{w}_{\ell}^e$ and using numerical integration.
Note that this is the only occurrence, where numerical quadratures rules are needed. All the other quantities can indeed be computed exactly using the degrees of freedom, see \cite{TVEM_Helmholtz_num}.

\subsection{Numerical experiments} \label{subsection numerical experiments}
In this section, we employ the method~\eqref{VEM} to approximate the solution to~\eqref{weak monolithic formulation} in three different test cases, using the notation of Section~\ref{section fluid-fluid interface problem}:
\begin{itemize}
\item \texttt{test case~1}: given an incoming traveling plane wave with $\thetainc \ge \thetacrit$, this wave is partially reflected at the interface~$\Gamma$ and a plane wave is transmitted in the subdomain $\Omega_2$;
\item \texttt{test case~2}: given an incoming traveling plane wave with $\thetainc < \thetacrit$, the wave is completely reflected and evanescent modes appear in~$\Omega_2$;
\item \texttt{test case~3}: we consider the same situation as in \texttt{test case~1}, but employing here meshes with elements that are cut by the interface~$\Gamma$.
\end{itemize}
Note that for all the test cases, the exact solution is known in closed form. In fact, assuming that $\uinc$ is an incoming traveling plane wave with angle $\thetainc$ and wave number $\k_1$, i.e.,
\begin{equation*}
\uinc(\x) := \exp(i k_1 \d \cdot \x),\quad \d := (\cos(\thetainc), \sin(\thetainc) ),
\end{equation*}
the solution to the global problem~\eqref{weak monolithic formulation} is given by
\begin{equation} \label{exact solution}
u := 
\begin{cases}
\uinc + u_R & \text{in }\Omega_1\\
u_T & \text{in } \Omega_2.\\
\end{cases}
\end{equation}
The reflected and the transmitted waves, respectively, can be expressed as
\begin{equation} \label{reflected transmitted waves}
u_R(x,y) := R \exp(i k_1 \d \cdot (x,\, -y)),\quad u_T(x,y) := T \exp(ik_2 (K_1 x + K_2  y)),
\end{equation}
where the coefficients $R$, $T$, $K_1$ and $K_2$ are computed by employing the transmission conditions in~\eqref{strong transmission problem}: 
\begin{equation*}
K_1 := \k_1/k_2 \cos (\thetainc), \quad K_2 := \sqrt{1-k_1^2/k_2^2 \cos^2(\thetainc)}, \quad R := \frac{k_1 \sin(\thetainc)-k_2 K_2}{k_1 \sin(\thetainc)+k_2 K_2},\quad T:=1+R.
\end{equation*}
Since an explicit representation of the numerical solution $\un$ is not available in closed form inside each element, it is not possible to compute the (exact) $H^1$ and $L^2$ discretization errors directly.
Instead, as usually done in VEM, we compute the approximate relative errors
\begin{equation} \label{rel_errors}
\frac{\lVert u-\Pi \un \rVert_{1,k,\taun}}{\lVert u \rVert_{1,k,\Omega}}, \quad
\frac{\lVert u-\Pi \un \rVert_{0,\taun}}{\lVert u \rVert_{0,\Omega}}, 
\end{equation} 
where ${\Pi}_{|\E}(\vh)  =\PiEp (\vh)$ for all $\vh \in \VhE$, for all $\E \in \taun^1$, and ${\Pi}_{|\E}(\vh)=\PiEpqtilde(\vh)$ for all $\vh \in \VhE$, for all $\E \in \taun^2$, are the local projectors defined in~\eqref{Pi projector}.

As stabilization $\SE(\cdot,\cdot)$ in~\eqref{local discrete sesquilinear form}, we employ
\begin{equation}\label{modified D-recipe}
\SE(\un,\vn)= \sum_{s=1}^{n_\E} \sum_{\ell=1}^{\widehat{p}_{e_s}} \aE(\Pi \varphi_{s,\ell}, \Pi \varphi_{s,\ell}) \dof_{s,\ell}(\un) \overline{\dof_{s,\ell}(\vn)},
\end{equation}
where $\Pi^\E$ is either $\PiEp$ or $\PiEpqtilde$, depending on $\E$.
Such a stabilization was introduced and discussed in~\cite{ncTVEM_Helmholtz,TVEM_Helmholtz_num} and can be seen as a generalization of the \emph{diagonal recipe} stabilization in~\cite{VEM3Dbasic, fetishVEM, fetishVEM3D}.

\subsubsection{\texttt{Test case 1} (incoming plane wave with $\thetainc>\thetacrit$)} \label{subsubsection testcase1}
We first consider the test case of an incoming plane wave with incident angle $\thetainc>\thetacrit$. In this case, reflection and transmission of plane waves take place.

As refraction indices, we pick $n_1=2$ and $n_2=1$. Accordingly with~\eqref{critical angle}, the critical angle is $\thetacrit=60^\circ$.
We consider~$\thetainc=75^\circ$ and~$k=7$, i.e., local wave numbers~$\k_1=14$ and~$\k_2=7$. The exact solution is given in~\eqref{exact solution} and its real part is depicted in Figure~\ref{fig:exact_sol_testcase1}.
\begin{figure}[h]
\begin{center}
\begin{minipage}{0.48\textwidth} 
\includegraphics[width=\textwidth]{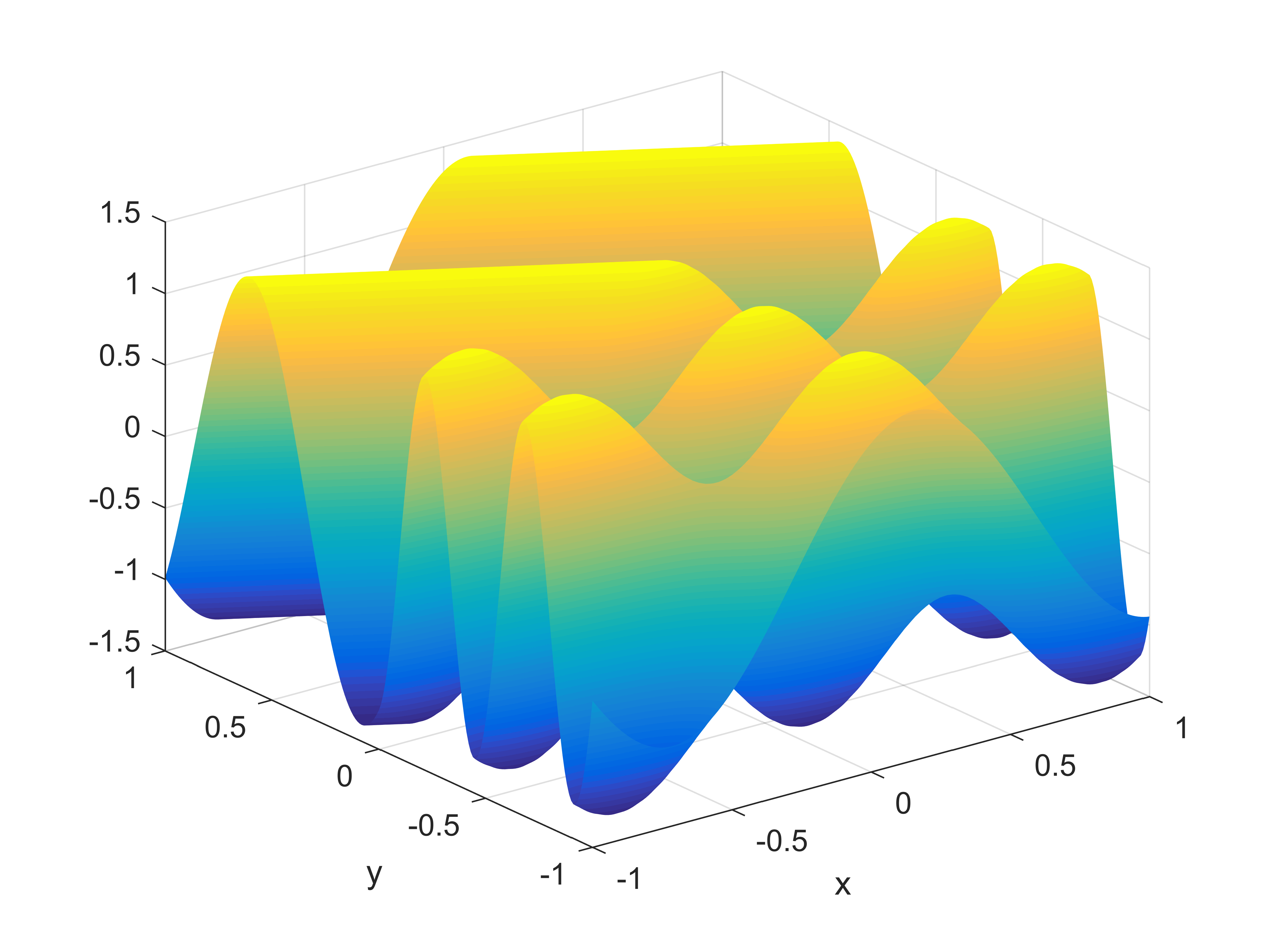}
\end{minipage}
\hfill
\begin{minipage}{0.48\textwidth}
\includegraphics[width=\textwidth]{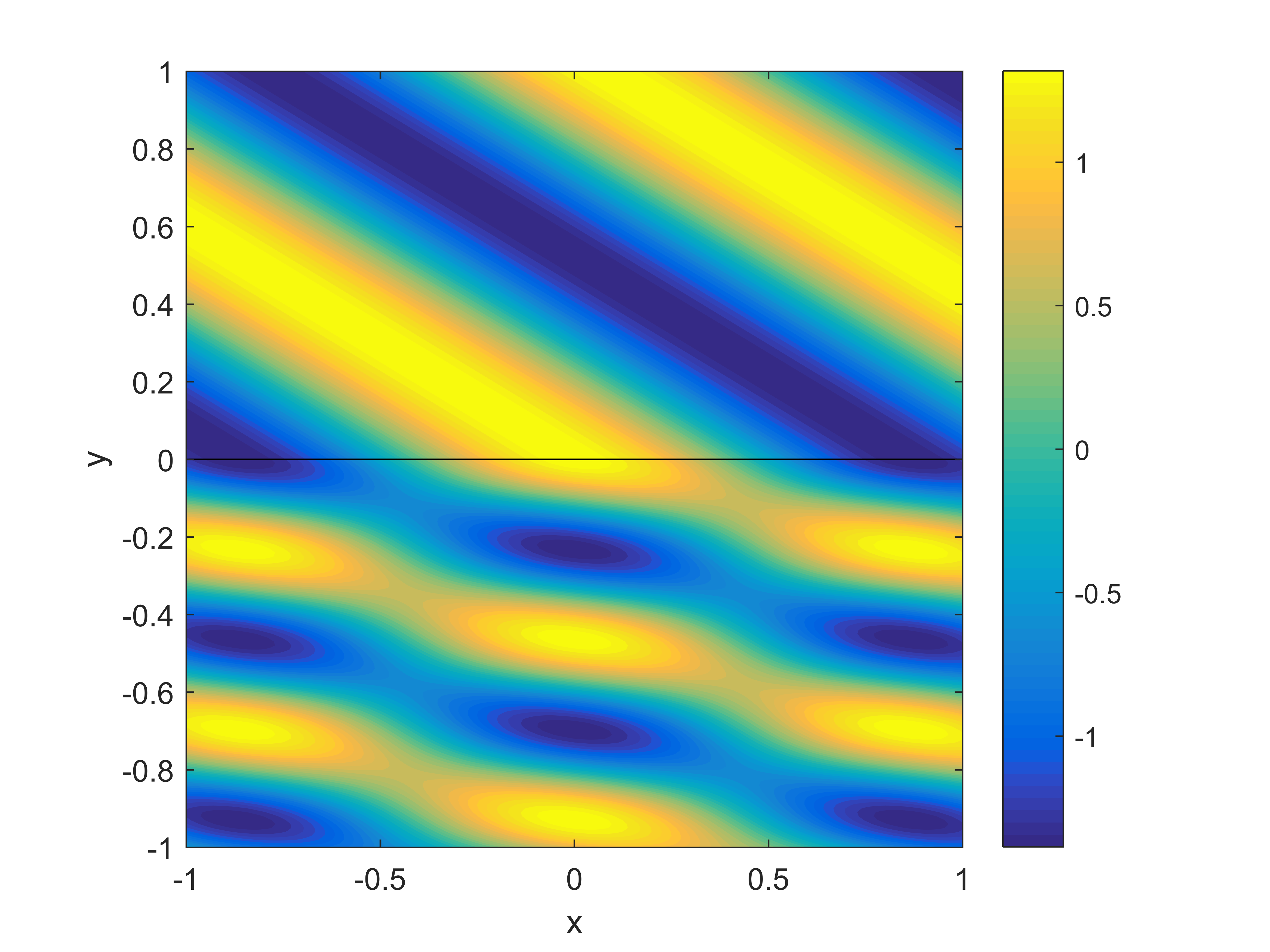}
\end{minipage}
\end{center}
\caption{Real part of the exact solution~$u$ given by~\eqref{exact solution} with $\k=7$, $n_1=2$, $n_2=1$, and $\thetainc=75^\circ$. \textit{Left}: surface plot. \textit{Right}: contour plot, where the black line indicates the interface~$\Gamma$.}
\label{fig:exact_sol_testcase1} 
\end{figure}

We study the~$\h$- and $\p$-versions of the method for the problem~\eqref{strong transmission problem}, where the impedance datum~$\g$ is computed accordingly with the exact analytical solution.
Inside each subdomains $\Omega_1$ and $\Omega_2$ only plane waves with the same set of equidistributed directions are employed.
In the following, we will always write $q_1$, $q_2$ and $\widetilde{q}_2$ when the effective plane/evanescent wave degrees do not vary elementwise within each subdomain.

For the $h$-version, we study the behaviour of the error curves for different values of $q_1$ and $q_2$, namely $q_1=q_2=4$, $q_1=q_2=6$, and $q_1=12$ with $q_2=6$.
Recall that the numbers of plane waves in $\Omega_1$ and $\Omega_2$, respectively, are given by $p_1=2q_1+1$ and $p_2=2q_2+1$.
Since no evanescent modes are expected to appear in $\Omega_2$ and the transmitted solution is a plane wave, we do not add evanescent waves to the local spaces, i.e., we take $\widetilde{q}_2=0$.
We employ sequences of standard regular Cartesian meshes and Voronoi meshes (reflected across the~$x$- and the $y$-axes), see Figure~\ref{fig:voronoi meshes}.
The results are depicted in Figure~\ref{fig:testcase1h}. 

\begin{figure}[H]
\begin{center}
\begin{minipage}{0.3\textwidth} 
\includegraphics[width=\textwidth]{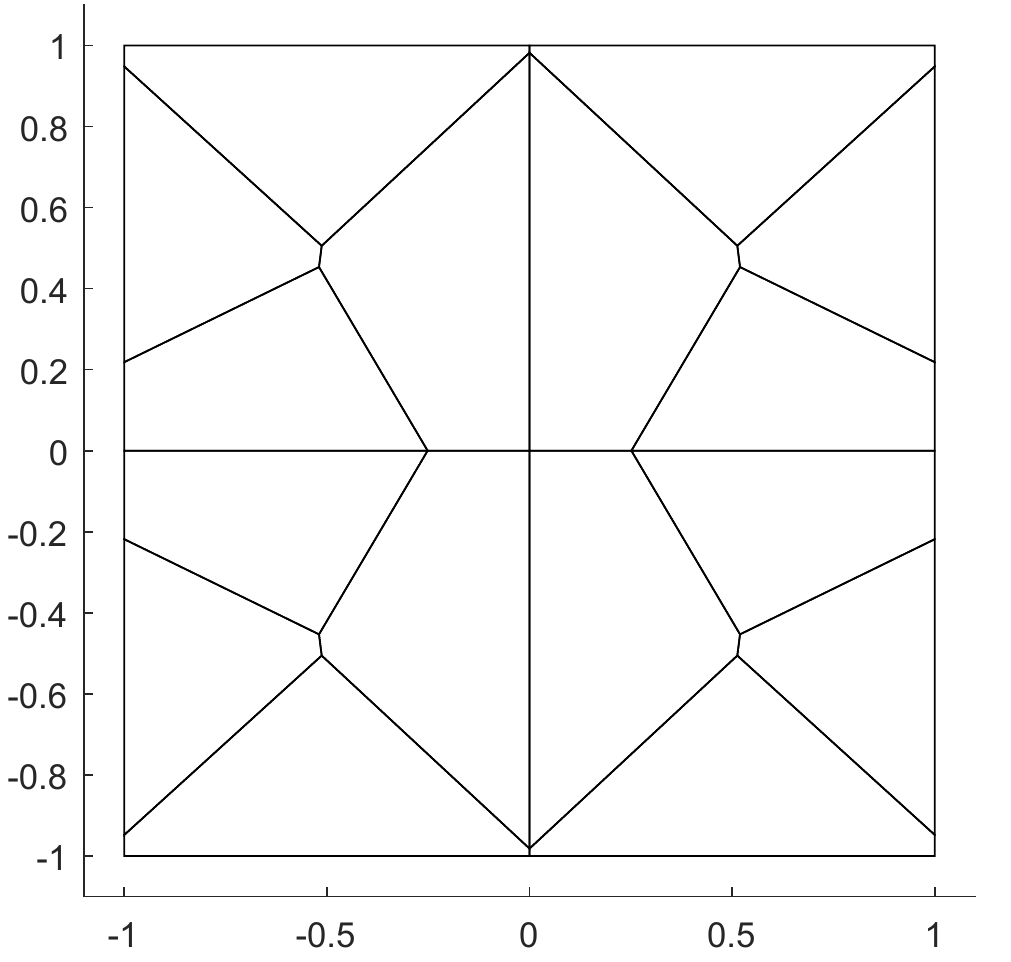}
\end{minipage}
\hfill
\begin{minipage}{0.3\textwidth}
\includegraphics[width=\textwidth]{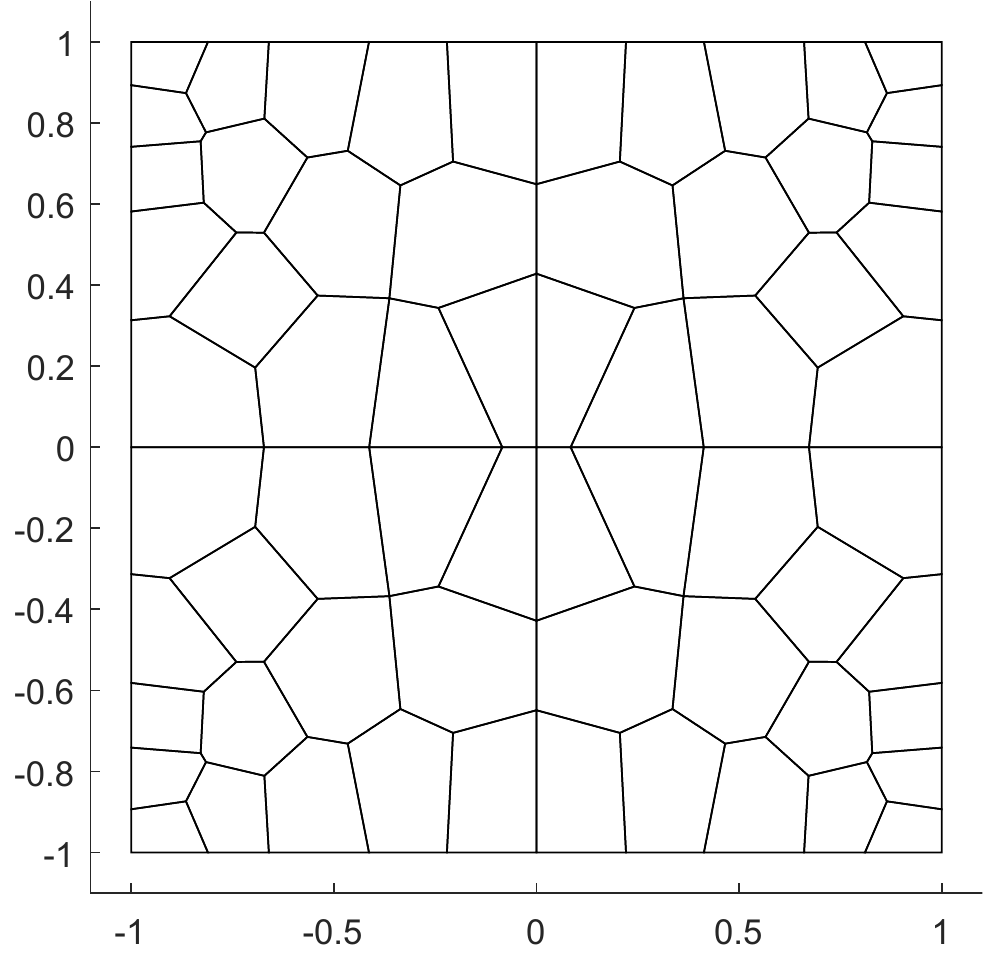}
\end{minipage}
\hfill
\begin{minipage}{0.3\textwidth}
\includegraphics[width=\textwidth]{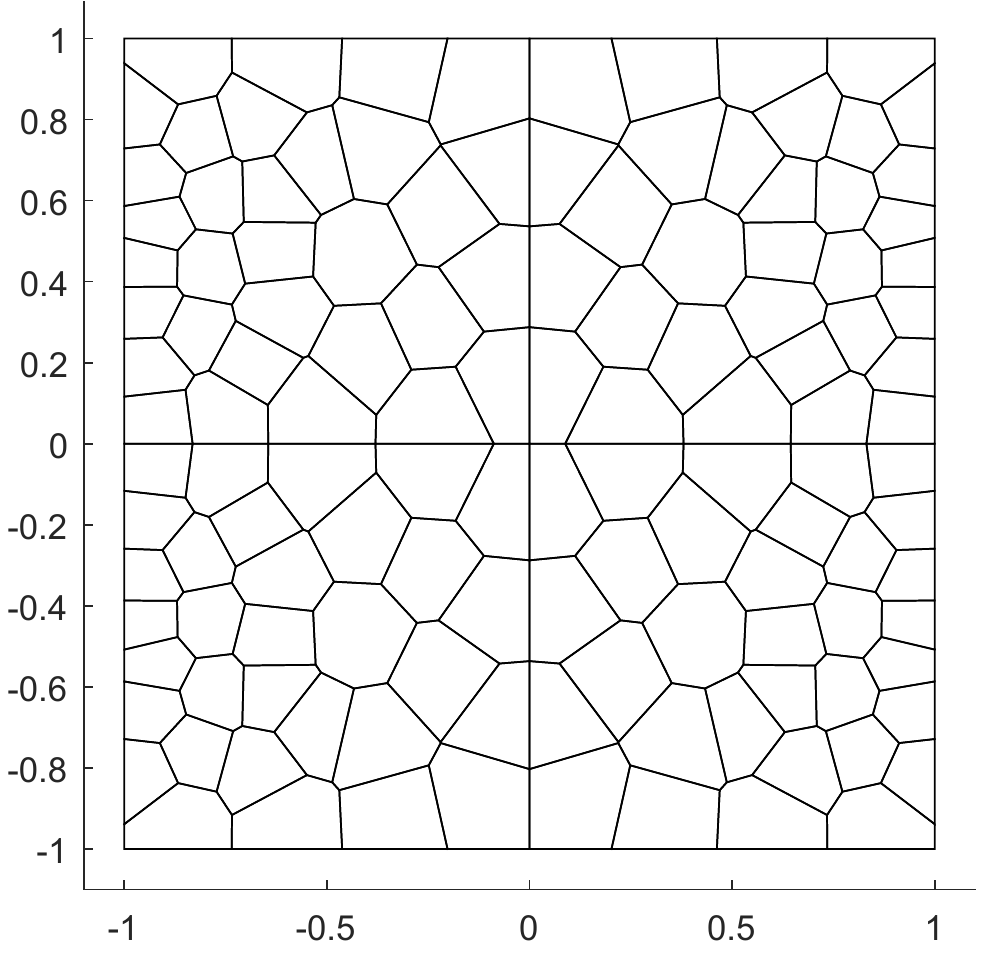}
\end{minipage}
\end{center}
\caption{Voronoi meshes (reflected across the $x$- and the $y$-axes) with 16, 64, and 128 elements, from \textit{left} to \textit{right}.}
\label{fig:voronoi meshes} 
\end{figure}
We observe algebraic convergence in terms of the minimal effective degree $\min\{q_1,q_2\}$.
The rates for the $H^1$ and $L^2$ errors are approximatively given by $\min\{q_1,q_2\}$ and $\min\{q_1,q_2\}+1$, respectively.
Further, when using the Voronoi meshes, the curves are not as straight as in the Cartesian case. This can be explained by the presence of very small edges and of elements with different sizes.

\begin{figure}[h]
\begin{center}
\begin{minipage}{0.48\textwidth} 
\includegraphics[width=\textwidth]{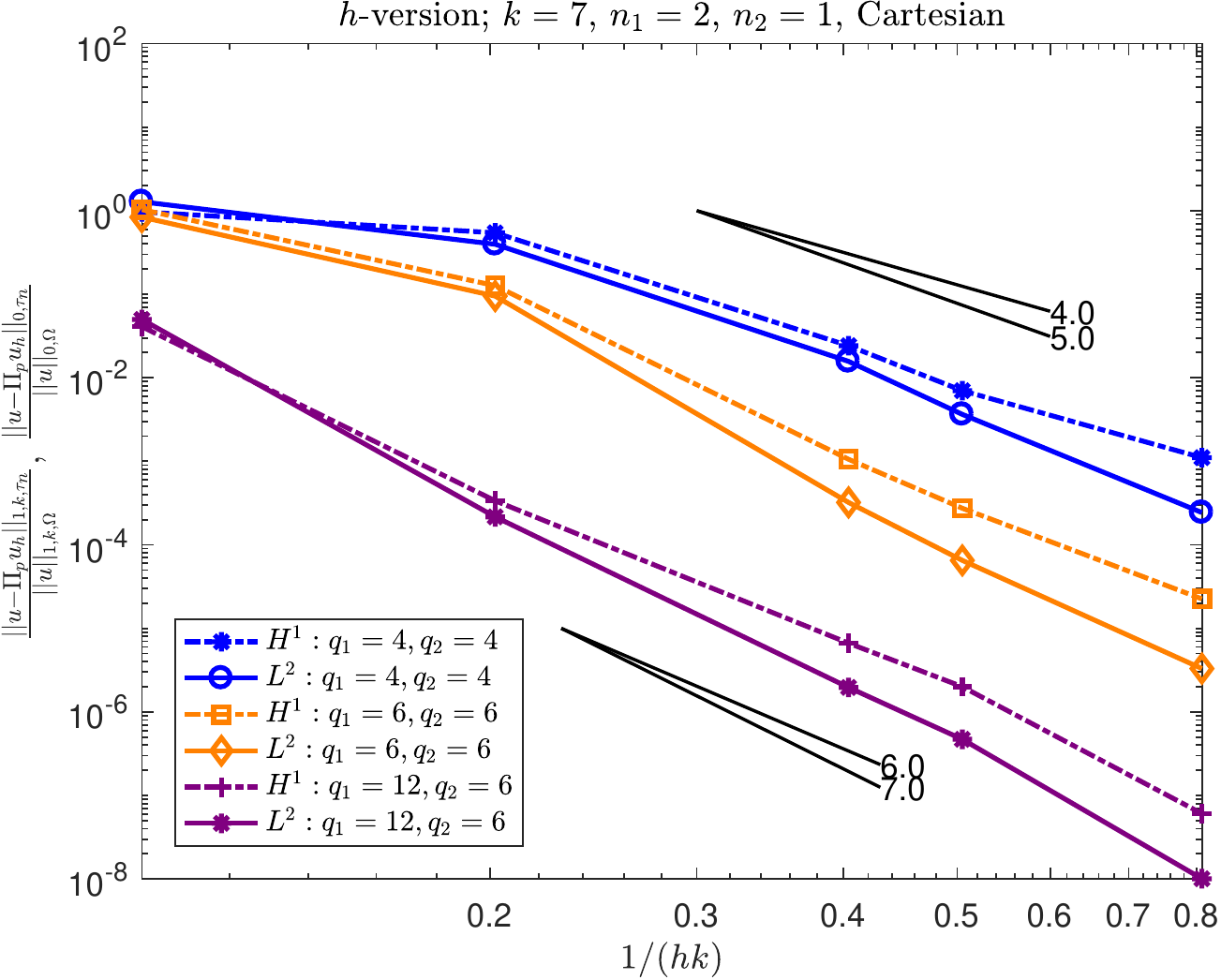}
\end{minipage}
\hfill
\begin{minipage}{0.48\textwidth}
\includegraphics[width=\textwidth]{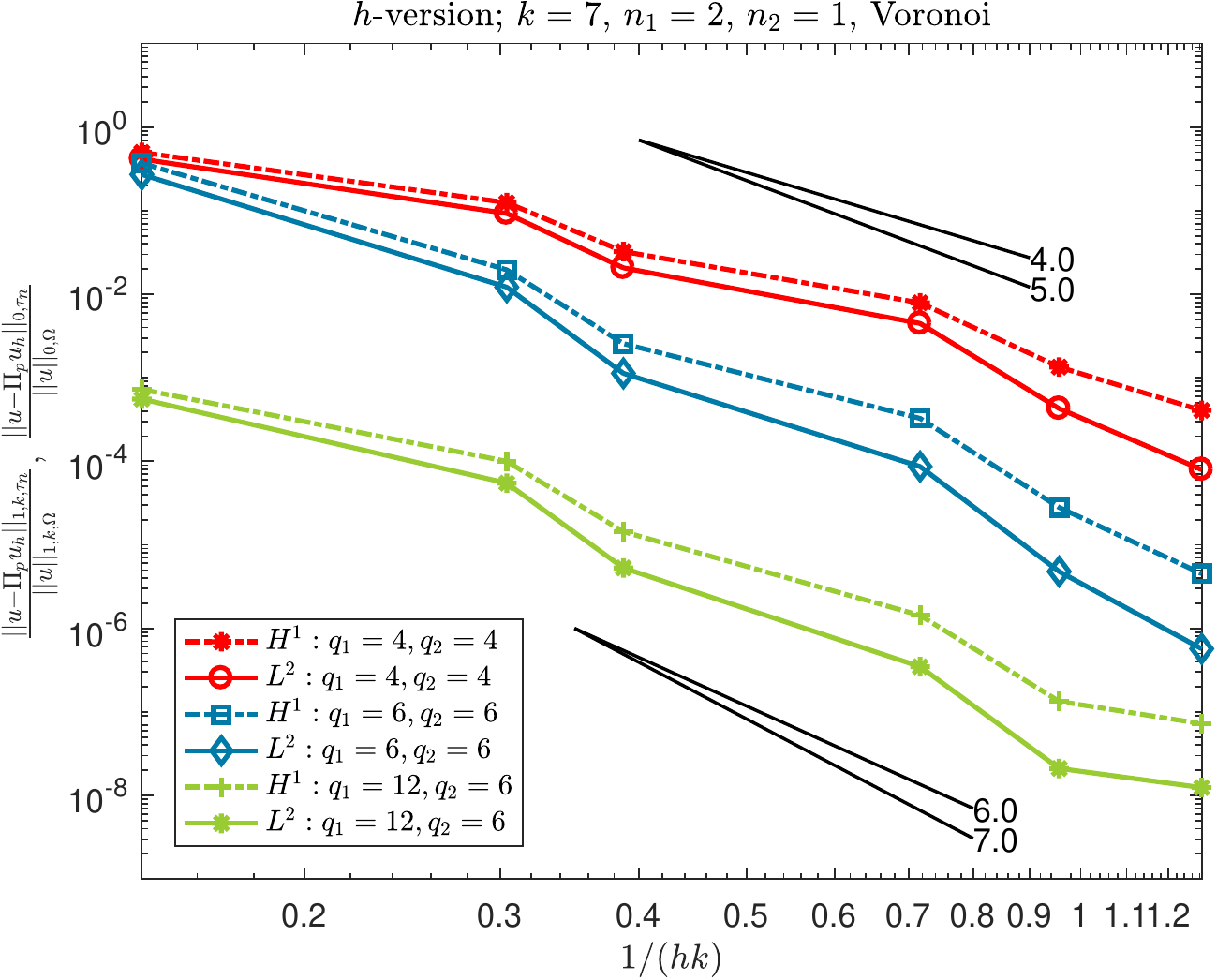}
\end{minipage}
\end{center}
\caption{$h$-version of the method for $u$ in~\eqref{exact solution} with $k=7$, $n_1=2$, $n_2=1$, and $\thetainc=75^\circ$ on a sequence of regular Cartesian meshes (\textit{left}) and a sequence of Voronoi meshes as in Figure~\ref{fig:voronoi meshes} (\textit{right}). The relative errors are computed accordingly with~\eqref{rel_errors}.}
\label{fig:testcase1h}
\end{figure}

Next, we investigate the $\p$-version of the method. To this end, we fix a regular Cartesian mesh and the Voronoi mesh in Figure~\ref{fig:voronoi meshes} with~$64$ elements.
We vary the effective degrees $q_1$ and $q_2$, and study the behaviour for the cases $q_1=q_2$ and $q_1=2q_2$. The error plots are displayed in Figure~\ref{fig:testcase1p}.

\begin{figure}[h]
\begin{center}
\begin{minipage}{0.48\textwidth} 
\includegraphics[width=\textwidth]{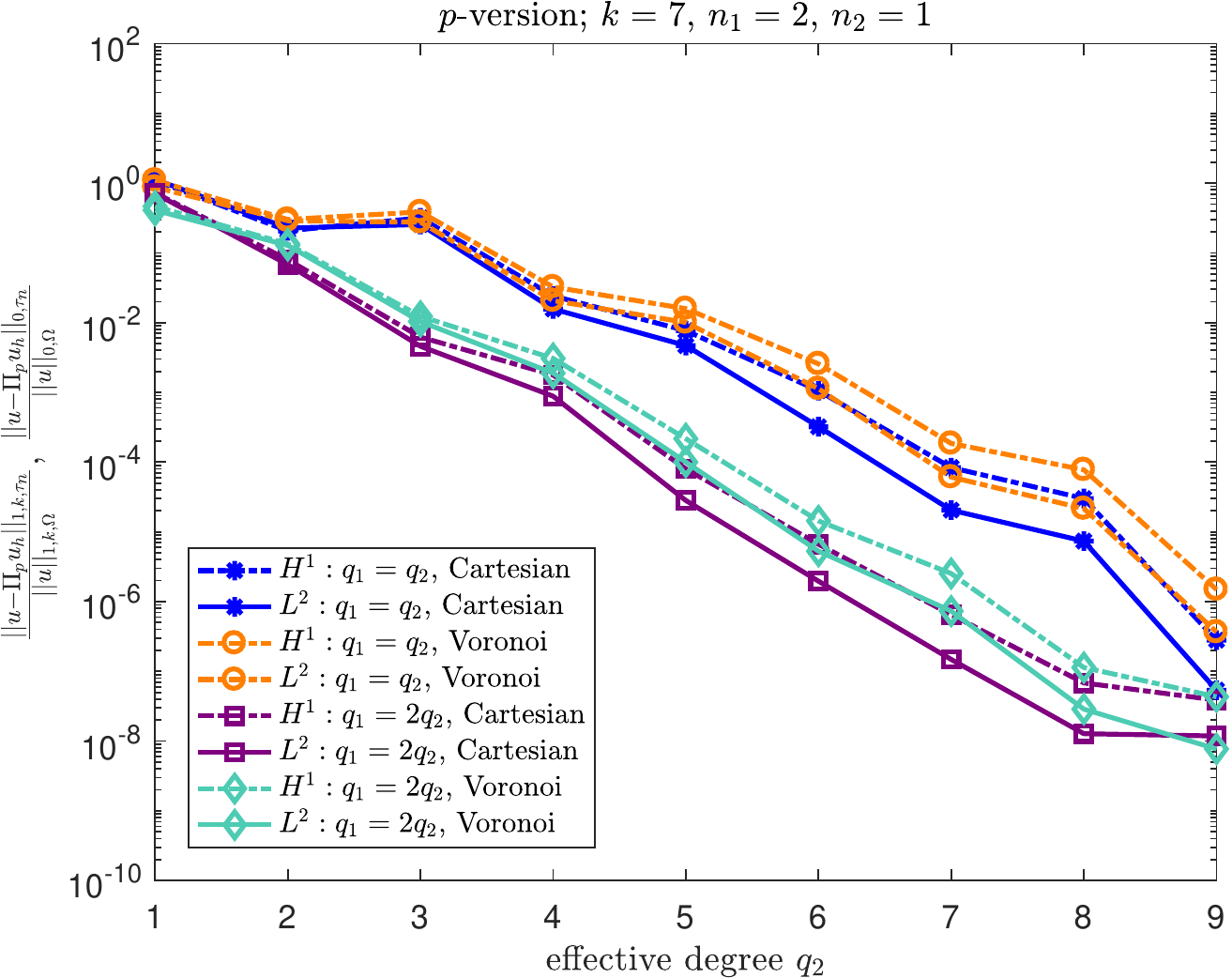}
\end{minipage}
\hfill
\begin{minipage}{0.48\textwidth}
\includegraphics[width=\textwidth]{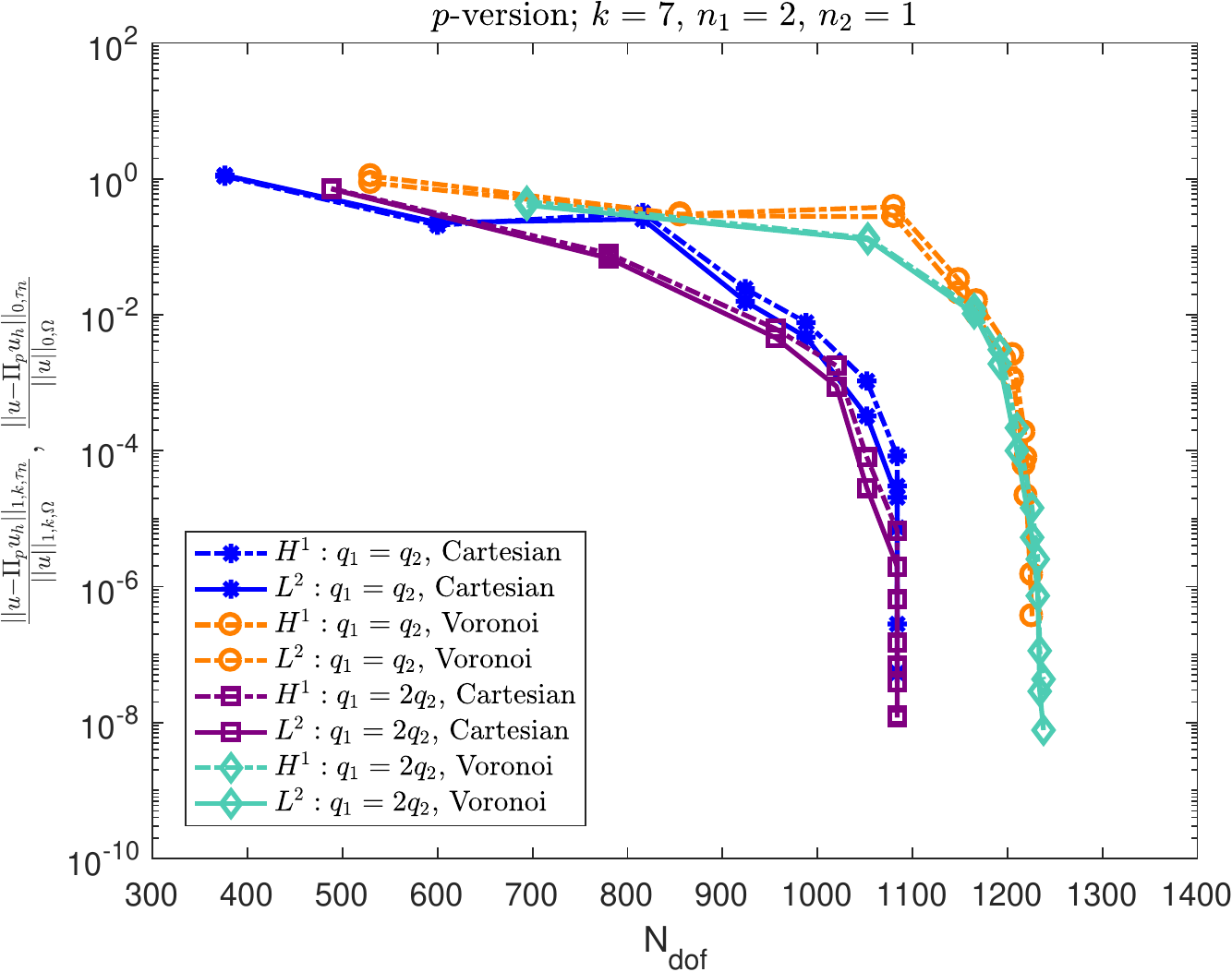}
\end{minipage}
\end{center}
\caption{$p$-version of the method for $u$ in~\eqref{exact solution} with $k=7$, $n_1=2$, $n_2=1$, and $\thetainc=75^\circ$ on a regular Cartesian mesh and the Voronoi mesh in Figure~\ref{fig:voronoi meshes} with 64 elements each. The relative errors are computed accordingly with~\eqref{rel_errors}. \textit{Left}: relative bulk errors against $q_2$. \textit{Right}: relative bulk errors against the number of degrees of freedom.}
\label{fig:testcase1p}
\end{figure}
We observe exponential convergence with respect to the effective degree~$q_2$, where the slope of the error curves is basically the same for $q_1=q_2$ and $q_1=2q_2$, but the accuracy is a few orders higher in the latter case.
The behaviour depicted in Figure~\ref{fig:testcase1p} (right) is instead a consequence of the or\-tho\-go\-na\-li\-za\-tion-and-fil\-te\-ring process in Algorithm~\ref{algorithm orthog process}.
In fact, when increasing~$\p$, the growth of the number of degrees of freedom slows down; this results in a convergence rate which is effectively more than exponential.
Interestingly, in the last $\p$-refinements, the error seems to tend to zero even without an increase of the number of degrees of freedom.

It is worth to underline that the exponential convergence of the $\p$-version is  expected from the fact that we have considered so far meshes that are conforming with respect to the interface~$\Gamma$
and that the exact solution is piecewise analytic on the two subdomains~$\Omega_1$ and~$\Omega_2$.

In Section~\ref{subsubsection nonconforming meshes}, we will investigate the performance of the method employing meshes that are nonconforming with respect to~$\Gamma$.

\subsubsection{\texttt{Test case 2} (incoming plane wave with $\thetainc<\thetacrit$)} \label{subsubsection testcase2}
Here, we fix the incident angle of the incoming wave $\thetainc<\thetacrit$. This leads to total reflection of the plane wave at the interface~$\Gamma$; evanescent modes occur in~$\Omega_2$.
Since the evanescent modes are characterized by an exponential decay, the method could benefit from adding special functions which decay exponentially as well, that is, evanescent waves.
To this purpose, inspired by~\cite{tezaur2008discontinuous, luostari2013improvements}, we compare the method when only plane waves are used in $\Omega_2$ with the case when also evanescent waves in $\Omega_2$ are added. Similarly as above, we investigate the~$\h$- and~$\p$-versions. 

We pick $\k=7$, $n_1=2$ and $n_2=1$, as before, and the incoming angle $\thetainc=50^\circ$. The real part of the corresponding exact solution computed as in~\eqref{exact solution} is plotted in Figure~\ref{fig:exact_sol_testcase2}.

\begin{figure}[h]
\begin{center}
\begin{minipage}{0.48\textwidth} 
\includegraphics[width=\textwidth]{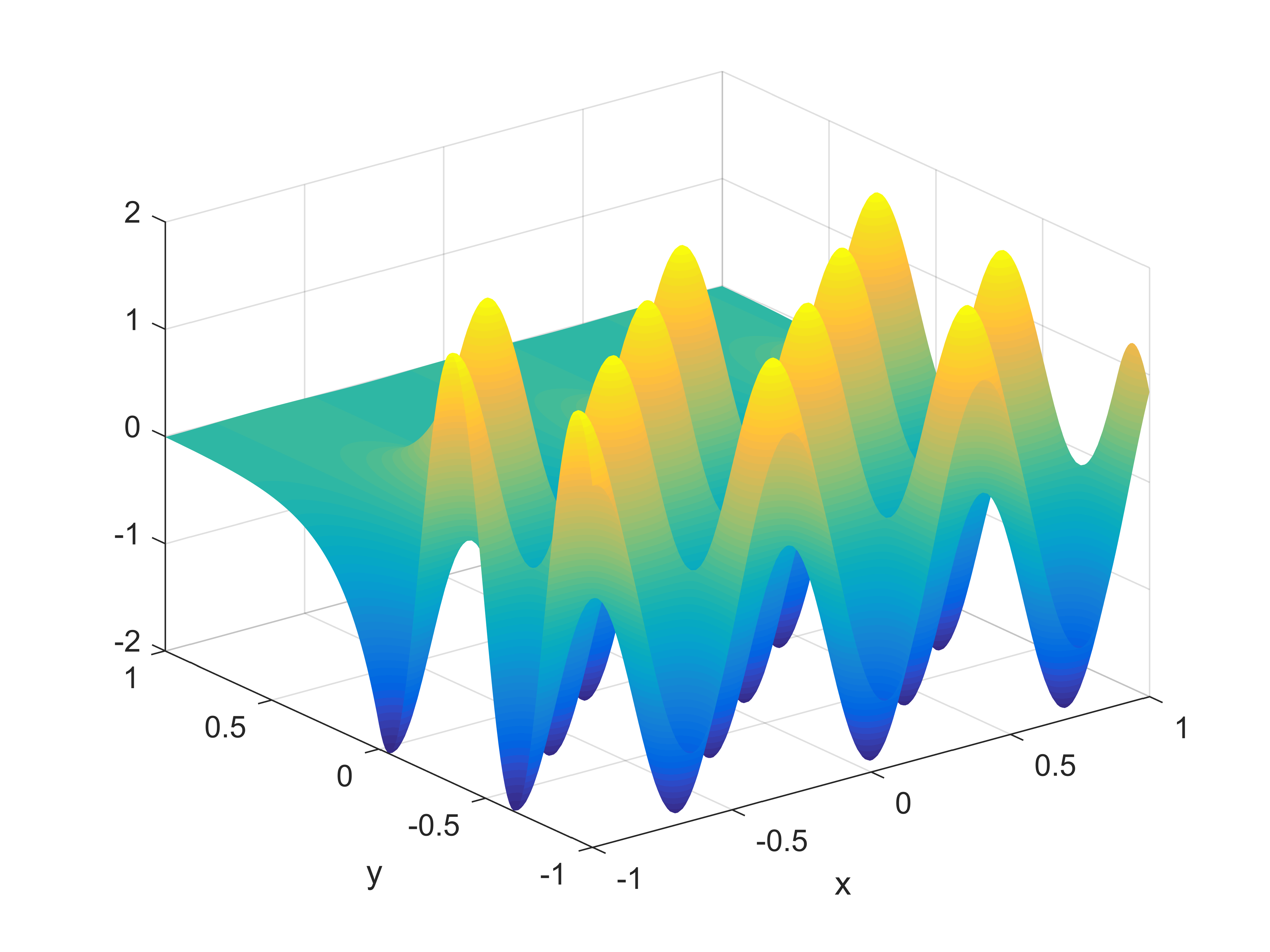}
\end{minipage}
\hfill
\begin{minipage}{0.48\textwidth}
\includegraphics[width=\textwidth]{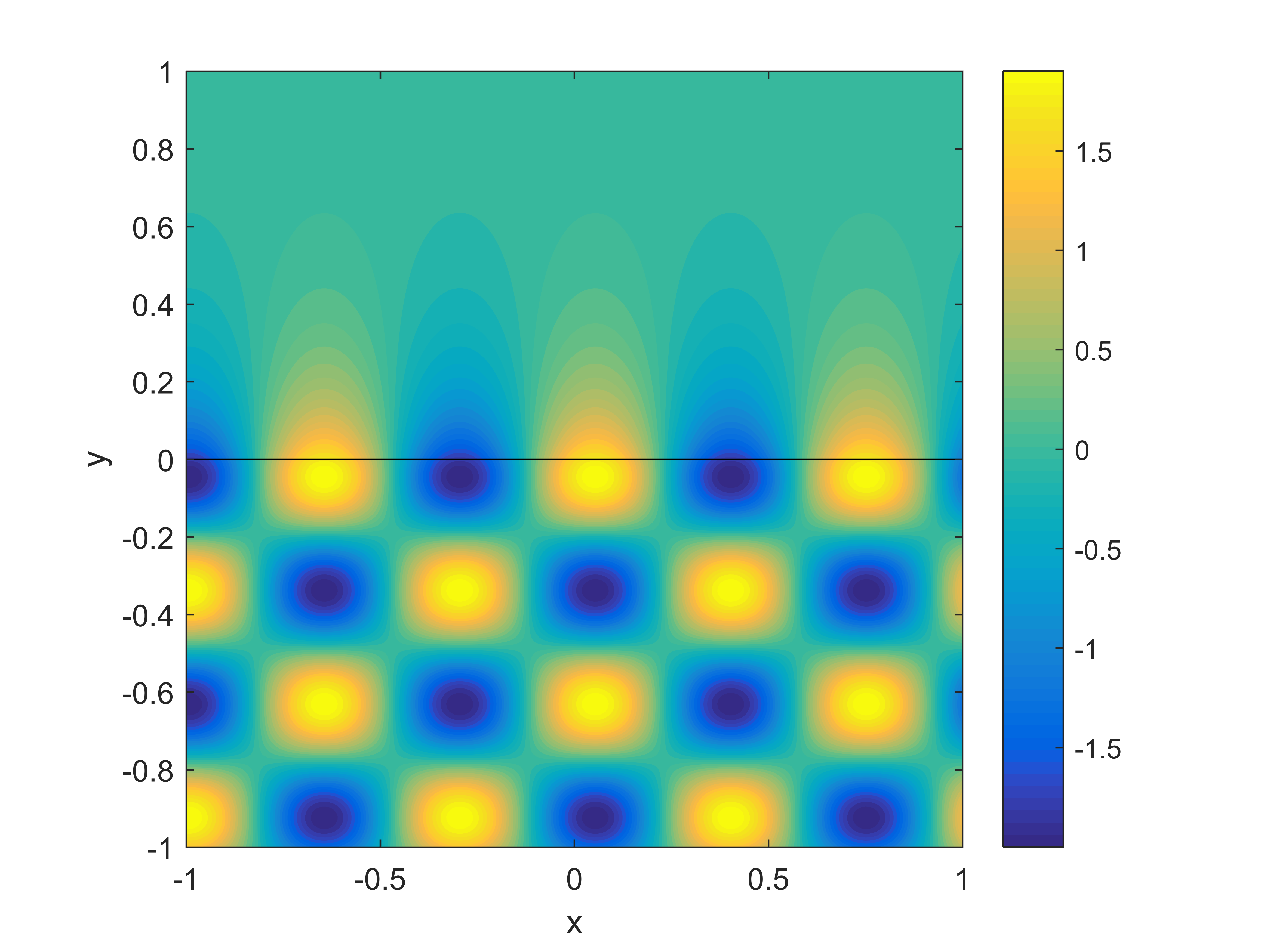}
\end{minipage}
\end{center}
\caption{Real part of the exact solution~$u$ given by~\eqref{exact solution} with $\k=7$, $n_1=2$, $n_2=1$, and $\thetainc=50^\circ$. \textit{Left}: surface plot. \textit{Right}: contour plot, where the black line indicates the interface~$\Gamma$.}
\label{fig:exact_sol_testcase2} 
\end{figure}

For the $h$-version, we assume once again that the effective plane/evanescent wave degree is the same for all elements within a subdomain. In $\Omega_1$, we take $\q_1=12$ (namely, 25 plane waves), whereas in $\Omega_2$ we consider
\begin{itemize}
\item $q_2=6$ and $\widetilde{q}_2=0$, i.e., $13$ plane waves and $0$ evanescent waves;
\item $q_2=5$ and $\widetilde{q}_2=1$, i.e., $11$ plane waves and $2$ evanescent waves;
\item $q_2=4$ and $\widetilde{q}_2=2$, i.e., $9$ plane waves and $4$ evanescent waves;
\item $q_2=0$ and $\widetilde{q}_2=6$, i.e., $0$ plane waves and $12$ evanescent waves.
\end{itemize}

Note that we do not choose $q_1=q_2+\widetilde{q}_2$ on purpose, since in this case the discretization error in~$\Omega_1$ dominates that in~$\Omega_2$ due to the higher local wave number.
For this reason, we picked~$q_1$ equal to the double of~$q_2+\qtilde_2$.

We employ the same meshes as for the $\h$-version in \texttt{test case~1}. The results are plotted in Figure~\ref{fig:testcase2h}.
As already indicated in~\cite[Section 4]{luostari2013improvements}, by adding evanescent waves to the local spaces, the order of convergence of the method is not changed, but the accuracy is improved by a multiplicative factor. 
We also underline that the convergence deteriorates when the error becomes sufficiently small (typically around $10^{-8}$).
This effect can be traced back to the ill-conditioning haunting the wave based methods and which can not be totally removed by Algorithm~\ref{algorithm orthog process}.

\begin{figure}[h]
\begin{center}
\begin{minipage}{0.48\textwidth} 
\includegraphics[width=\textwidth]{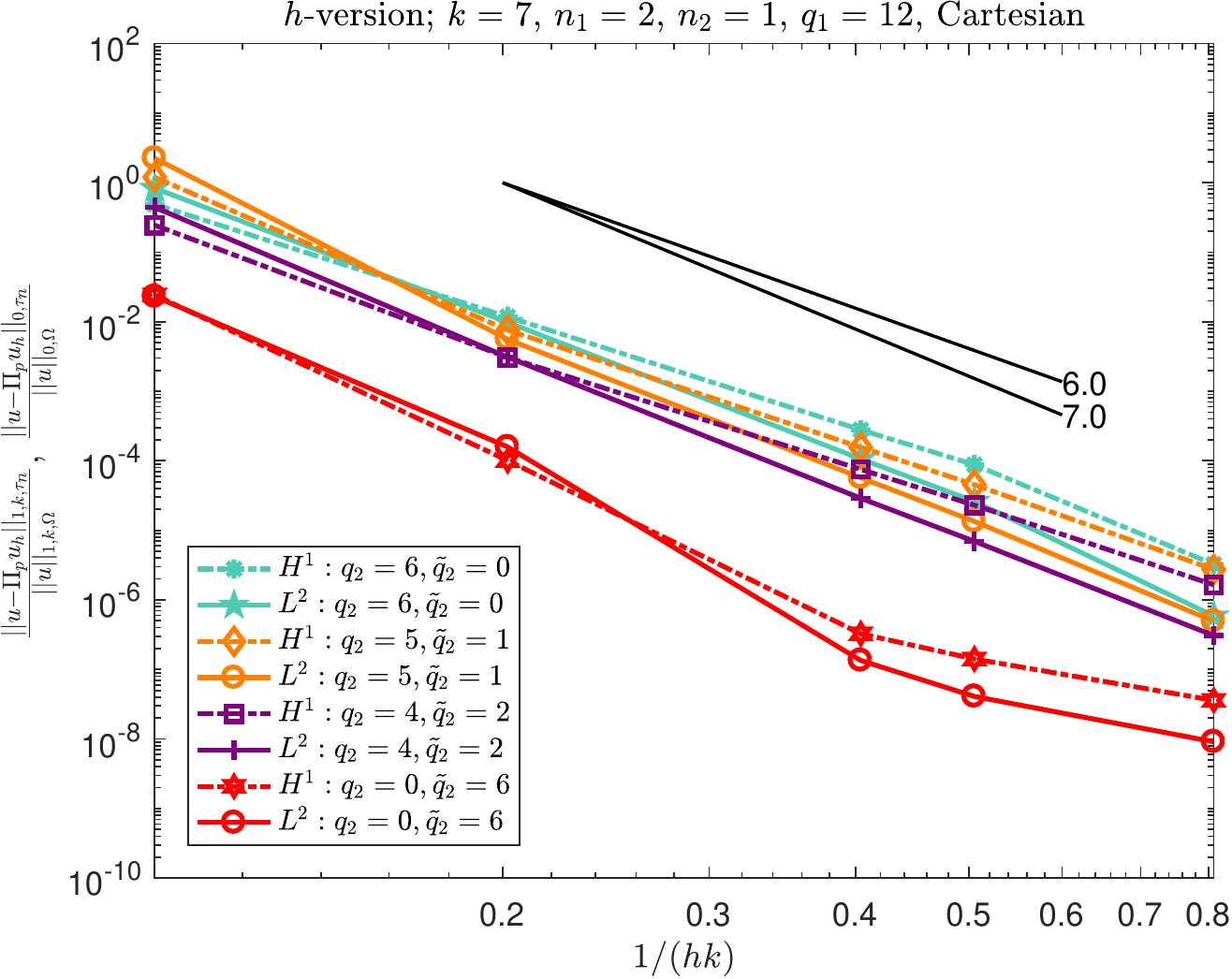}
\end{minipage}
\hfill
\begin{minipage}{0.48\textwidth}
\includegraphics[width=\textwidth]{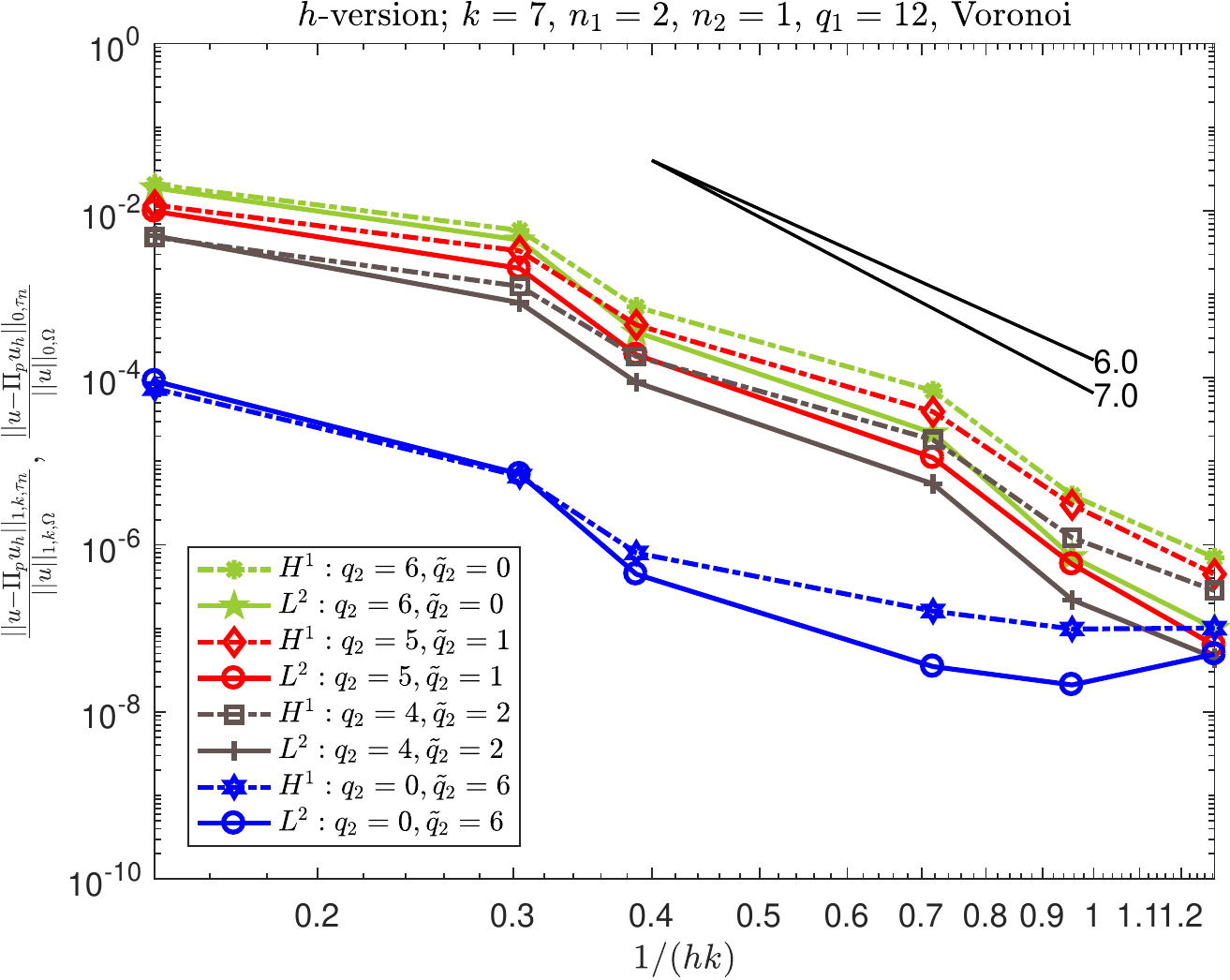}
\end{minipage}
\end{center}
\caption{$h$-version of the method for $u$ in~\eqref{exact solution} with $k=7$, $n_1=2$, $n_2=1$, $q_1=12$, and $\thetainc=50^\circ$ on a sequence of regular Cartesian meshes (\textit{left}) and a sequence of Voronoi meshes as in Figure~\ref{fig:voronoi meshes} (\textit{right}). The relative errors are computed accordingly with~\eqref{rel_errors}.}
\label{fig:testcase2h}
\end{figure}

Regarding the $p$-version, we fix, as before, the Voronoi mesh in Figure~\ref{fig:voronoi meshes} with 64 elements. This time we assume that $q_1=2(q_2+\widetilde{q}_2)$. We consider
\begin{itemize}
\item $\widetilde{q}_2=0$ and increase~$q_2$;
\item $\widetilde{q}_2=1$ and increase~$q_2$;
\item $\widetilde{q}_2=2$ and increase~$q_2$;
\item $q_2=0$ and increase~$\widetilde{q}_2$.
\end{itemize}
The error plots are shown in Figure~\ref{fig:testcase2p}. Similar results are obtained when using a regular Cartesian mesh with 64 elements; for this reason, we omit them.
As before, we observe exponential convergence in terms of the sum of the effective degrees $\q_2+\widetilde{q}_2$, where the accuracy of the method is again improved when evanescent waves are contained in the approximation spaces in $\Omega_2$.
The best performance is achieved when only evanescent waves are used in $\Omega_2$. 

\begin{figure}[h]
\begin{center}
\begin{minipage}{0.48\textwidth} 
\includegraphics[width=\textwidth]{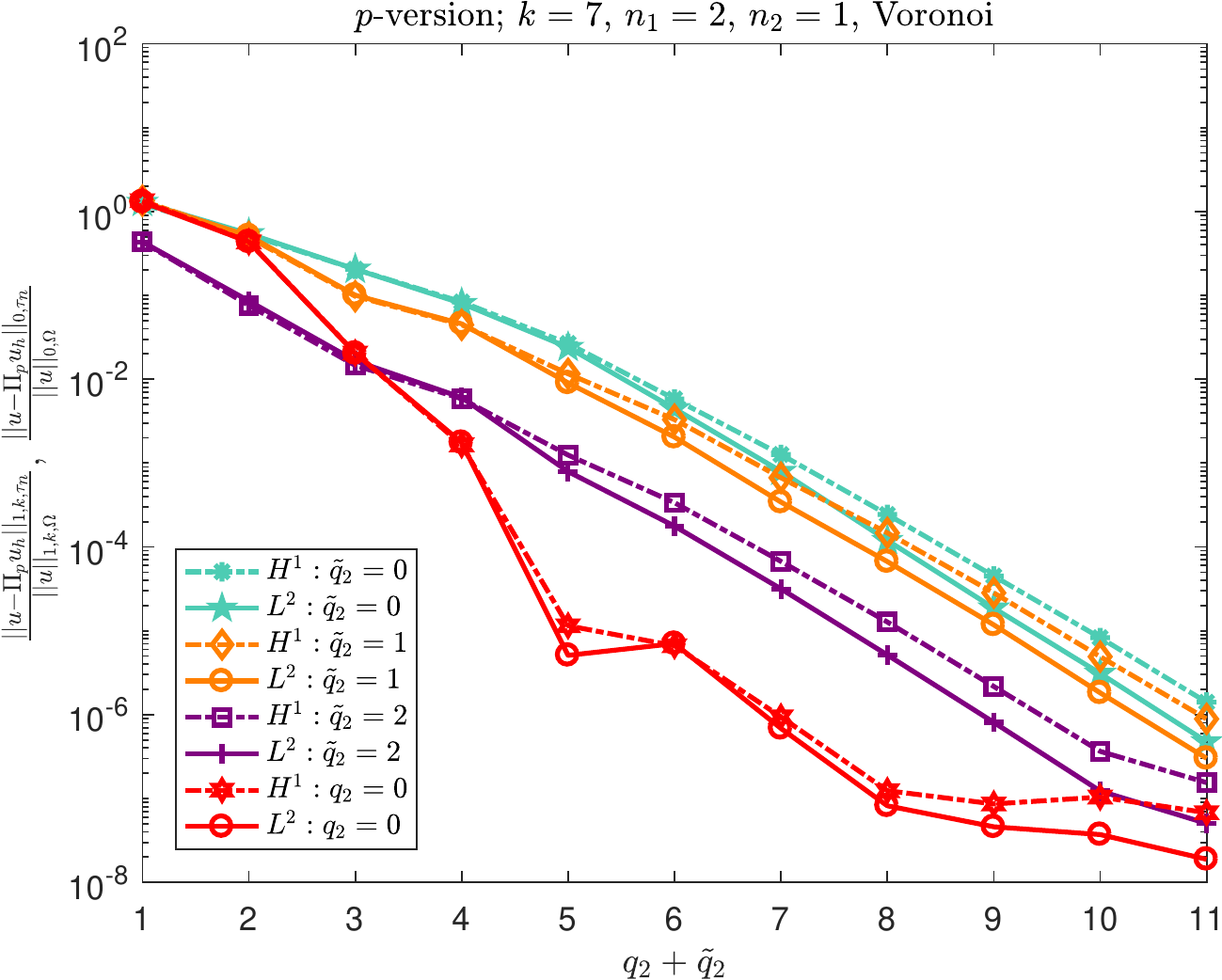}
\end{minipage}
\hfill
\begin{minipage}{0.48\textwidth}
\includegraphics[width=\textwidth]{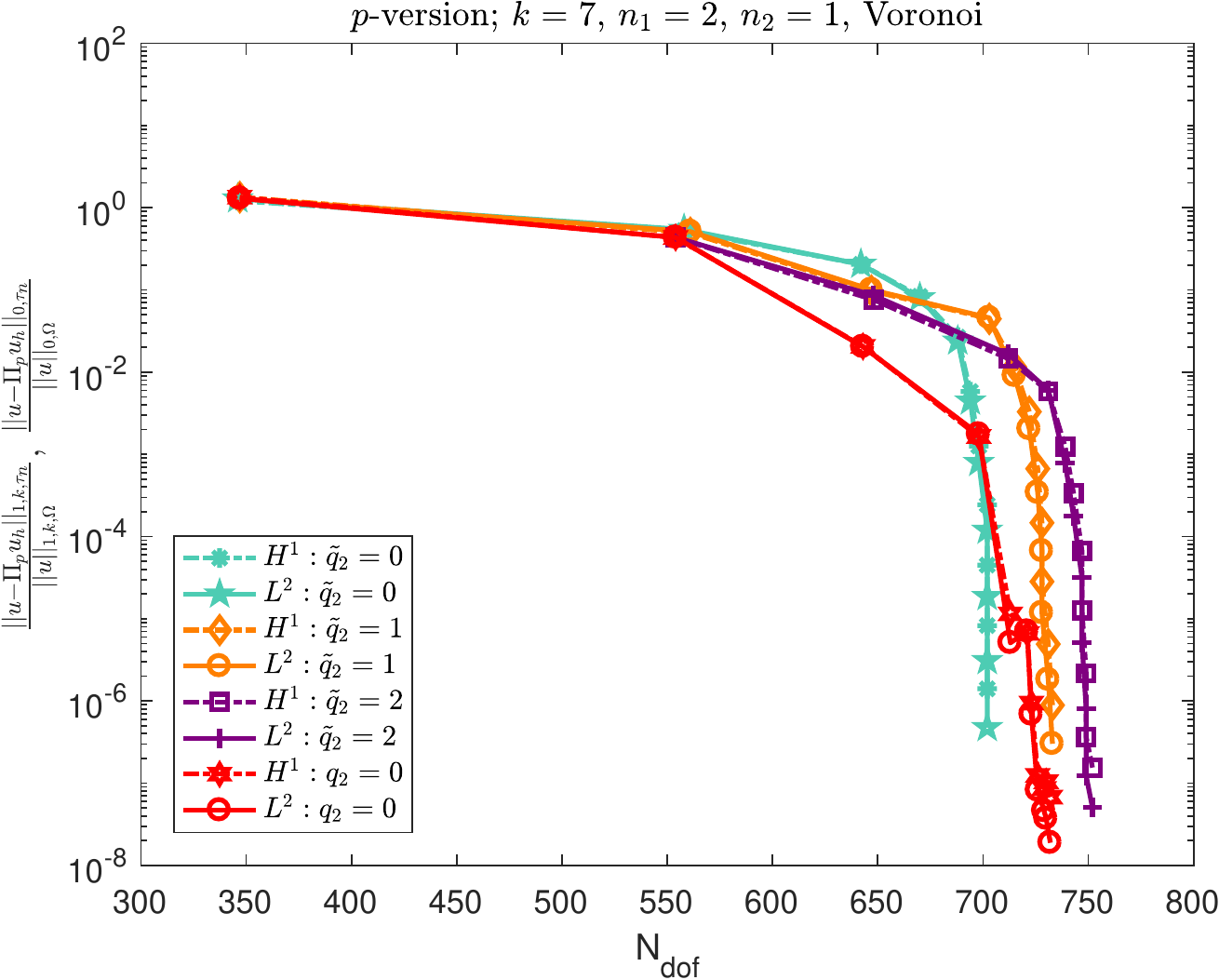}
\end{minipage}
\end{center}
\caption{$p$-version (effective degrees $q_1$, $q_2$ and $\widetilde{q}_2$ with $q_1=2(q_2+\widetilde{q}_2)$) of the method for $u$ in~\eqref{exact solution} with $k=7$, $n_1=2$, $n_2=1$, and $\thetainc=50^\circ$ on the Voronoi mesh with 64 elements in Figure~\ref{fig:voronoi meshes}.
The relative errors are computed accordingly with~\eqref{rel_errors}. \textit{Left}: relative bulk errors against $q_2+\widetilde{q}_2$. \textit{Right}: relative bulk errors against the number of degrees of freedom.}
\label{fig:testcase2p}
\end{figure}

%
%
%
\subsubsection{\texttt{Test case 3} (nonconforming meshes and the $\h\p$-version)} \label{subsubsection nonconforming meshes}
So far, we have employed sequences of meshes that are conforming with respect to the interface~$\Gamma$, that is, every~$\E$ in~$\taun$ is contained either in~$\Omega_1$ or in~$\Omega_2$.
The advantage of this choice is that since the explicit solution~\eqref{exact solution} is piecewise analytic, the~$\h$- and the~$\p$-versions of the method have optimal order of convergence.
In particular, the $\p$-version results in exponential convergence as highlighted in Figures~\ref{fig:testcase1p} and~\ref{fig:testcase2p}.
Such an exponential convergence is however in terms of the number and not in terms of the square root of the number of degrees of freedom. This is due to the Trefftz nature of the method.

We want to investigate now how the method can be tuned to address the case where some elements of the mesh are cut by the interface~$\Gamma$.
This situation can be of interest in the following situations:
\begin{itemize}
\item the interface~$\Gamma$ is curvilinear and one does not want to resort to curvilinear VEM~\cite{beiraorussovacca_curvedVEM}; in this case, some polygonal elements necessarily cut~$\Gamma$;
\item assuming that the parameter~$\kcal$ is subject to uncertainty, e.g. it is piecewise constant over subdomains with stochastic boundaries,
one could proceed by reduced basis techniques starting from a very coarse mesh, and then, perform adaptive mesh and space refinements.
\end{itemize}
The first issue that has to be faced is the definition of the local spaces over the elements~$\E$ in~$\taun$ such that $\E ^{\circ} \cap \Gamma \ne \emptyset$.
Since on such elements, the wave number~$\kcal$ takes two different values, namely~$\k_1$ and~$\k_2$,
we propose to fix the local spaces $\VhE$ defined as in~\eqref{local spaces}, with wave number either given by
the maximum between~$k_1$ and~$\k_2$ (i.e., $\k_1$), or the average of~$k_1$ and~$\k_2$. In both cases, the resulting method~\eqref{VEM} is not Trefftz anymore.

For the forthcoming numerical tests, we focus for simplicity on the exact solution to \texttt{test case~1}, i.e., when the incident angle is larger than the critical angle. Furthermore, we do not employ evanescent waves
and only considers here the case where the average of the wave number is chosen in the elements abutting~$\Gamma$.
Note that, slightly worse results are obtained when taking the maximum between the two wave numbers.
\medskip

Another issue to cope with is that, since the solution is analytic over the two subdomains~$\Omega_1$ and~$\Omega_2$, but not over the complete domain~$\Omega$,
the standard~$\h$- and $\p$-versions of the method may not converge or converge suboptimally when employing nonconforming meshes.

In order to overcome such a problem, we will employ $\h\p$-refinements, that is, we will construct virtual element spaces based on polygonal meshes that are graded geometrically towards the interface~$\Gamma$
and have local effective degrees possibly varying from element to element. In particular, we will resort to both isotropic and anisotropic mesh refinements.
\medskip

The remainder of this section is organized as follows.
In Sections~\ref{paragraph ISO} and~\ref{paragraph ANISO}, we describe the construction of virtual element spaces with elementwise variable effective degree
on geometrically graded meshes employing isotropic and anisotropic mesh refinements, respectively.
Next, in Section~\ref{paragraph h VS hpISO}, we present numerical experiments, where we compare the~$\h$- and the $\h\p$-versions (with isotropic mesh refinements) of the method.
Finally, a comparison between $\h\p$-isotropic and anisotropic mesh refinements is discussed in Section~\ref{paragraph ISO vs ANISO}.

\paragraph{$\h\p$-virtual element spaces on isotropic geometrically refined meshes.} \label{paragraph ISO}
The scope of the present section is to introduce geometric isotropic mesh refinements towards the interface~$\Gamma$ and the associated $\h\p$-virtual element spaces.

First, we define the concept of layers of a mesh~$\taun$. To this purpose, we assume that a mesh $\taun$ consists of~$n+1$ layers.
The $0$-th layer $L_n^0$ is the set of all polygons abutting the interface~$\Gamma$, whereas the other layers are defined by induction as
\[
L^n_\ell := \left\{ \E_1 \in \taun \mid \overline {\E_1} \cap \overline {\E_2}\ne \emptyset \text{ for some }\E_2\in L^n_{\ell-1},\, \E_1 \not \subseteq \cup_{j=0}^{\ell-1} L_j^n   \right\} \quad \forall \ell =1, \dots, n.
\]
We say that $\{\taun\}_{n}$ is a sequence of isotropic geometrically graded meshes $(i)$ if~$\mathcal T_{n+1}$ is obtained by starting from~$\taun$ and refining only the elements in the layer~$L^n_0$,
and $(ii)$ if there exists a grading parameter~$\sigma \in (0,1)$ such that
\begin{equation} \label{sigma ISO}
\hE \approx \sigma^{n-\ell} \quad \text{if } \E \in L_\ell^n.
\end{equation}
In words, such isotropic geometrically graded meshes are characterized by small elements abutting the interface and elements enlarging geometrically when the distance from~$\Gamma$ increases.
We assume that all the elements have bounded aspect ratio.

Next, we define $\h\p$-virtual element spaces over such meshes and we introduce two types of distributions of the effective degrees.
To this end, we first define the dimension of plane and evanescent waves over edges;
denoting by~$\pbold\in \mathbb N^{\card(\taun)}$ the vector of the local effective degrees, the vector~$\pboldepsilon\in \mathbb N ^{\card(\En)}$,
i.e., the vector of the dimensions of the spaces $\PWtilde(\e)$ in~\eqref{edge plane-evanescent waves}, is given by
\[
(\pboldepsilon)_{|\ell}=
\begin{cases}
\max(\pbold_i,\pbold_j) 	& \text{if } \e_\ell \in \En^I \text{ and } \e_\ell\subseteq \partial \E_i \cap \partial \E_j \\
\pbold_i 				& \text{if } \e_\ell \in \En^B \text{ and } \e_\ell \subset \partial \E_i \\
\end{cases}
\quad \forall \ell = 1,\dots, \card(\En).
\]
In the numerical experiments, we will employ, for some positive parameter~$\mu$, both a uniform (increasing) distribution of the degrees of freedom
\begin{equation} \label{uniform p}
\pbold_j = \lceil \mu(n+1) \rceil \quad \forall j=1,\dots, \card(\taun),
\end{equation}
as well as a graded one:
\begin{equation} \label{hp p}
\pbold_j = \lceil  \mu (\ell+1)  \rceil \text{ if } \E_j \in L^n_\ell \quad \forall j=1,\dots,\card (\taun).
\end{equation}
In~\eqref{uniform p} and~\eqref{hp p}, $\lceil \cdot \rceil$ denotes the ceiling function.
The latter approach is based on effective degrees growing together with the layer index. In fact, the singularity is approximated with the aid of small elements, whereas, the analytic part is approximated on large elements with high effective degrees.

In Figure~\ref{fig: hp mesh 1}, we depict  the first two meshes~$\mathcal T_1$ and~$\mathcal T_2$ (including the graded distribution~\eqref{hp p} of the effective degrees with~$\mu=1$)
of a sequence of isotropic geometrically graded meshes with grading parameter~$\sigma$ in~\eqref{sigma ISO} given by $1 / 3$.
\begin{figure}[H]
\centering
\begin{minipage}{0.30\textwidth}
\begin{center}
\begin{tikzpicture}[scale=2.5]
\draw[black, very thick, -] (0,0) -- (2,0) -- (2,2) -- (0, 2) -- (0,0);
\draw[red, dashed, thick, -] (0,1) -- (2,1);
\draw[black, very thick, -] (0, 4/3) -- (2, 4/3); \draw[black, very thick, -] (0, 2/3) -- (2, 2/3); \draw[black, very thick, -] (1, 2/3) -- (1, 4/3);
\draw(1, 2-4/13) node[black] {2}; \draw(1, 4/13) node[black] {2};
\draw(1/2, 1) node[black] {1}; \draw(3/2, 1) node[black] {1};
\draw(2-1/9, 1+1/20) node[red] {\tiny{$\Gamma$}};
\end{tikzpicture}
\end{center}
\end{minipage}
\quad \quad\quad \quad\quad \quad\quad \quad
\begin{minipage}{0.30\textwidth}
\begin{center}
\begin{tikzpicture}[scale=2.5]
\draw[black, very thick, -] (0,0) -- (2,0) -- (2,2) -- (0, 2) -- (0,0);
\draw[red, dashed, thick, -] (0,1) -- (2,1);
\draw[black, very thick, -] (0, 4/3) -- (2, 4/3); \draw[black, very thick, -] (0, 2/3) -- (2, 2/3); \draw[black, very thick, -] (1, 2/3) -- (1, 4/3);
\draw[black, very thick, -] (0, 1+1/9) -- (2, 1+1/9); \draw[black, very thick, -] (0, 1-1/9) -- (2, 1-1/9); \draw[black, very thick, -] (1/2, 1+1/9) -- (1/2, 1-1/9); \draw[black, very thick, -] (3/2, 1+1/9) -- (3/2, 1-1/9);
\draw(1, 2-4/13) node[black] {3}; \draw(1, 4/13) node[black] {3};
\draw(1/2, 3/4) node[black] {2}; \draw(3/2, 3/4) node[black] {2};
\draw(1/2, 5/4) node[black] {2}; \draw(3/2, 5/4) node[black] {2};
\draw(1/4, 1) node[black] {1}; \draw(3/4, 1) node[black] {1};
\draw(5/4, 1) node[black] {1}; \draw(7/4, 1) node[black] {1};
\draw(2-1/9, 1+1/20) node[red] {\tiny{$\Gamma$}};
\end{tikzpicture}
\end{center}
\end{minipage}
\caption{First two meshes~$\mathcal T_1$ and~$\mathcal T_2$ (including the graded distribution~\eqref{hp p} of the effective degrees, with~$\mu=1$) of a sequence of isotropic geometrically graded meshes. The grading parameter~$\sigma$ in~\eqref{sigma ISO} is~$1 / 3$.
The \red{dashed red} line denotes the interface~$\Gamma$.}
\label{fig: hp mesh 1}
\end{figure}
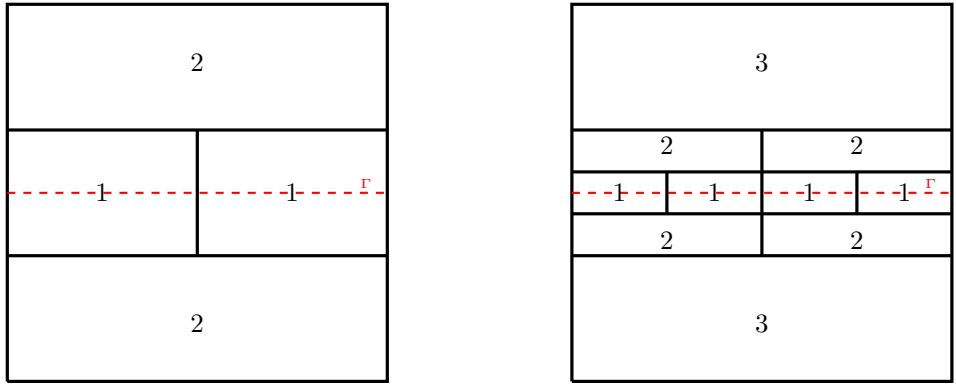

\paragraph{$\h\p$-virtual element spaces on anisotropic geometrically refined meshes.} \label{paragraph ANISO}
The scope of the present section is to describe anisotropic geometric mesh refinements towards the interface~$\Gamma$ and the associated $\h\p$-virtual element spaces.

The concept of layers of~$\taun$ is the same as in Section~\ref{paragraph ISO} and is therefore omitted here.
The difference from isotropic geometric mesh refinements is that,
given~$\E\in \taun$, and~$\h_{\E,1}$ and~$\h_{\E,2}$ the lengths of the edges of the rectangle of minimal perimeter bounding~$\E$ with edges parallel to~$\Gamma$ and its normal direction, respectively,
we say that $\{\taun\}_n$ is a sequence of anisotropic geometric mesh refinements if~$\mathcal T_{n+1}$ is obtained starting from~$\taun$ and refining only the elements in the layer~$L^n_0$, and if there exists a grading parameter~$\sigma\in (0,1)$ such that
\begin{equation} \label{sigma ANISO}
\h_{\E,2} \approx \sigma^{n-\ell} \quad \text{if } \E \in L_\ell^n, \quad \quad \h_{\E,1} \approx 1 \quad \quad \forall \E \in \taun.
\end{equation}
In words, we consider very thin elements in proximity of the interface~$\Gamma$ and larger elements elsewhere.

The reason why we also employ anisotropic mesh refinements is that the solution is singular only in the normal direction to~$\Gamma$ and not along the tangential one.
Thus, roughly speaking, it suffices to refine the mesh along the normal direction to~$\Gamma$.
Numerically, this results in a more effective approach for approximating edge singularities.
In fact, in the finite element framework, one gets exponential convergence in terms of the cubic root of the degrees of freedom (in the Trefftz setting, the cubic root becomes the square root,
see e.g. \cite{hmps_harmonicpolynomialsapproximationandTrefftzhpdgFEM, conformingHarmonicVEM, ncHVEM, TVEM_Helmholtz_num}),
whereas, with isotropic mesh refinements, one only obtains an algebraic rate of convergence.

Note that, for anisotropic meshes, we only employ the uniform distribution of the degrees of freedom~\eqref{uniform p}. The graded approach~\eqref{hp p} would not suffice for approximating the tangential part of the solution
(here, the elements have too long edges and therefore the method would not converge properly with very few degrees of freedom).

In Figure~\ref{fig: hp mesh 2}, we depict  the first two meshes~$\mathcal T_1$ and~$\mathcal T_2$ (including the uniform distribution of the effective degrees~\eqref{uniform p})
of a sequence of anisotropic geometrically graded meshes with grading parameter~$\sigma$ in~\eqref{sigma ANISO} given by $1 / 3$.
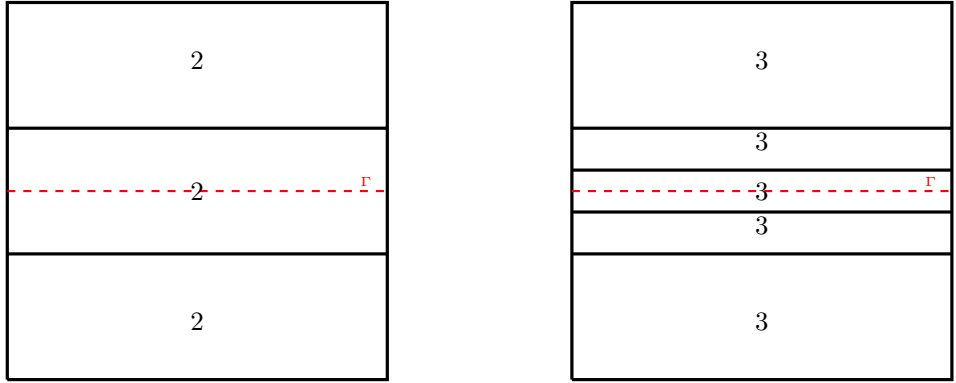
\begin{figure}[H]
\centering
\begin{minipage}{0.30\textwidth}
\begin{center}
\begin{tikzpicture}[scale=2.5]
\draw[black, very thick, -] (0,0) -- (2,0) -- (2,2) -- (0, 2) -- (0,0);
\draw[red, dashed, thick, -] (0,1) -- (2,1);
\draw[black, very thick, -] (0, 4/3) -- (2, 4/3); \draw[black, very thick, -] (0, 2/3) -- (2, 2/3); 
\draw(1, 2-4/13) node[black] {2}; \draw(1, 4/13) node[black] {2}; \draw(1, 1) node[black] {2};
\draw(2-1/9, 1+1/20) node[red] {\tiny{$\Gamma$}};
\end{tikzpicture}
\end{center}
\end{minipage}
\quad \quad\quad \quad\quad \quad\quad \quad
\begin{minipage}{0.30\textwidth}
\begin{center}
\begin{tikzpicture}[scale=2.5]
\draw[black, very thick, -] (0,0) -- (2,0) -- (2,2) -- (0, 2) -- (0,0);
\draw[red, dashed, thick, -] (0,1) -- (2,1);
\draw[black, very thick, -] (0, 4/3) -- (2, 4/3); \draw[black, very thick, -] (0, 2/3) -- (2, 2/3); 
\draw[black, very thick, -] (0, 4/3-2/9) -- (2, 4/3-2/9); \draw[black, very thick, -] (0, 2/3+2/9) -- (2, 2/3+2/9); 
\draw(1, 2-4/13) node[black] {3}; \draw(1, 4/13) node[black] {3}; \draw(1 , 4/3-1/14) node[black] {3}; \draw(1 , 4/3 - 4/9 -1/14) node[black] {3}; \draw(1 , 1) node[black] {3};
\draw(2-1/9, 1+1/20) node[red] {\tiny{$\Gamma$}};
\end{tikzpicture}
\end{center}
\end{minipage}
\caption{First two meshes~$\mathcal T_1$ and~$\mathcal T_2$ (including the uniform distribution of the effective degrees~\eqref{uniform p}) of a sequence of anisotropic geometrically graded meshes. The grading parameter~$\sigma$ in~\eqref{sigma ISO} is~$1 / 3$.
The \red{dashed red} line denotes the interface~$\Gamma$.}
\label{fig: hp mesh 2}
\end{figure}

\paragraph{Nonconforming meshes: comparison of the~$\h$- and the $\h\p$-isotropic versions.} \label{paragraph h VS hpISO}
In this section, we compare the $\h$-version of the method on sequences of uniform Cartesian meshes that are nonconforming with respect to the interface~$\Gamma$ employing~$\p=15$ plane wave directions,
and the $\h\p$-version of the method with isotropic geometrically graded mesh  as in Figure~\ref{fig: hp mesh 2}, endowed with both the uniform and the graded distribution of the effective degrees in~\eqref{uniform p} and~\eqref{hp p}, respectively.
In both cases, we pick~$\mu=1$, $2$, and $3$.

The results are displayed in Figure~\ref{fig: h vs hp ISO}, where we compare the number of degrees of freedom with the computable relative~$H^1$ and~$L^2$ errors in~\eqref{rel_errors}.

\begin{figure}[h]
\begin{center}
\begin{minipage}{0.7\textwidth} 
\includegraphics[width=\textwidth]{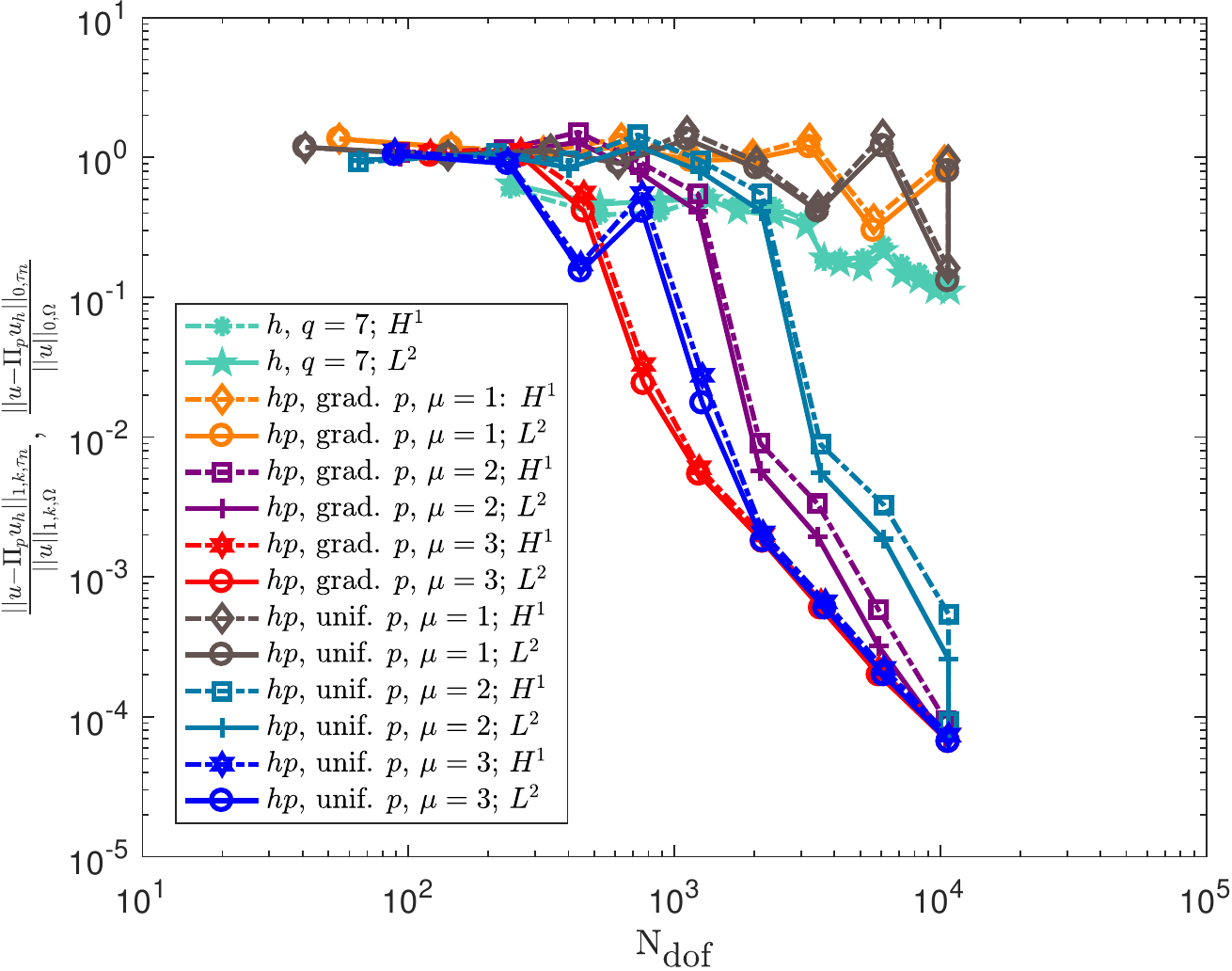}
\end{minipage}
\end{center}
\caption{$\h$-version employing nonconforming Cartesian meshes and $\h\p$-version with isotropic geometrically graded meshes, with grading parameter~$\sigma$ in~\eqref{sigma ISO} equal to~$1/3$, and~$\p=15$ plane waves on every element.
For the $\h\p$-spaces we consider both the uniform distribution of the degrees of freedom~\eqref{uniform p} and the graded one~\eqref{hp p}, with $\mu=1$, $2$, and~$3$.
The computable relative~$H^1$ and~$L^2$ errors in~\eqref{rel_errors} are plotted against the number of degrees of freedom.}
\label{fig: h vs hp ISO}
\end{figure}
From Figure~\ref{fig: h vs hp ISO}, we deduce that the $\h$-version converges poorly, due to the low Sobolev regularity of the solution.
The $\h\p$-version, on the other hand, performs much better. In particular, the choice of employing a graded distribution of the degrees of freedom seems to be the most effective.
It has to be underlined that in order to achieve the convergence regime, the parameter~$\mu$ in~\eqref{uniform p} and~\eqref{hp p} has to be picked sufficiently large, e.g.~$\mu=2$.

\paragraph{Nonconforming meshes: comparison of the $\h\p$-isotropic and anisotropic versions.} \label{paragraph ISO vs ANISO}
In this section, we compare the behaviour of the method for the case of $\h\p$-isotropic and anisotropic mesh refinements, using the meshes depicted in Figures~\ref{fig: hp mesh 1} and~\ref{fig: hp mesh 2}, respectively.
In particular, whereas in the isotropic case we only use the graded distribution~\eqref{hp p} (since we know from Section~\ref{paragraph h VS hpISO} that the uniform distribution~\eqref{uniform p} works slightly worse),
in the anisotropic case we employ a uniform distribution of the effective degrees~\eqref{uniform p}.
In both cases, we employ~$\mu=2$ and~$3$.

The results are presented in Figure~\ref{fig: hp ISO vs hp ANISO}, where we compare the computable relative~$H^1$ and~$L^2$ errors in~\eqref{rel_errors} in terms of the square root of the number of degrees of freedom.

\begin{figure}[h]
\begin{center}
\begin{minipage}{0.7\textwidth} 
\includegraphics[width=\textwidth]{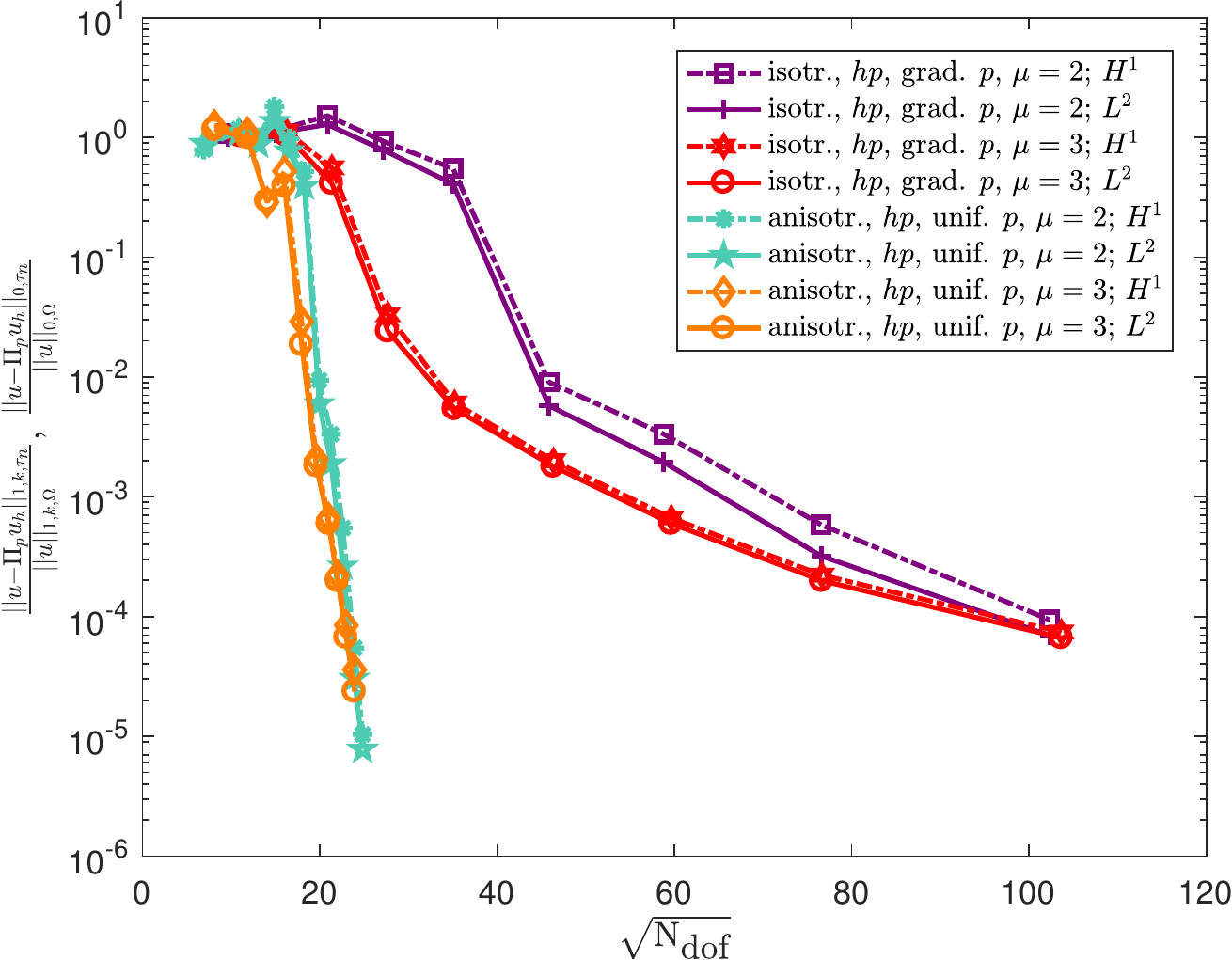}
\end{minipage}
\end{center}
\caption{$\h\p$-versions with geometrically isotropic and anisotropic graded meshes, with grading parameter~$\sigma$ in~\eqref{sigma ISO} equal to~$1/3$.
In the former case, we consider the graded distribution of the effective degrees~\eqref{hp p}, with $\mu=2$ and~$3$, whereas, in the latter, the uniform one~\eqref{uniform p} is applied.
We plot the computable relative~$H^1$ and~$L^2$ errors in~\eqref{rel_errors} against the square root of the degrees of freedom.}
\label{fig: hp ISO vs hp ANISO}
\end{figure}
From Figure~\ref{fig: hp ISO vs hp ANISO}, it is clear that employing anisotropic meshes leads to much better results.
Whilst exponential convergence in terms of the square root of the number of degrees of freedom is obtained for anisotropic meshes, the rate of convergence is only algebraic in the case of isotropic meshes.

\medskip
So far, we have employed the average of the two wave numbers as an ``artificial'' wave number on the elements abutting the interface~$\Gamma$.
In Figure~\ref{fig: different wn}, we present some numerical results for the $\h\p$-version of the method when also taking the maximum between the two of them.
We consider anisotropic mesh refinements and the uniform distribution of the effective degrees~\eqref{uniform p}, with~$\mu=2$ and~$3$.

\begin{figure}[h]
\begin{center}
\begin{minipage}{0.7\textwidth} 
\includegraphics[width=\textwidth]{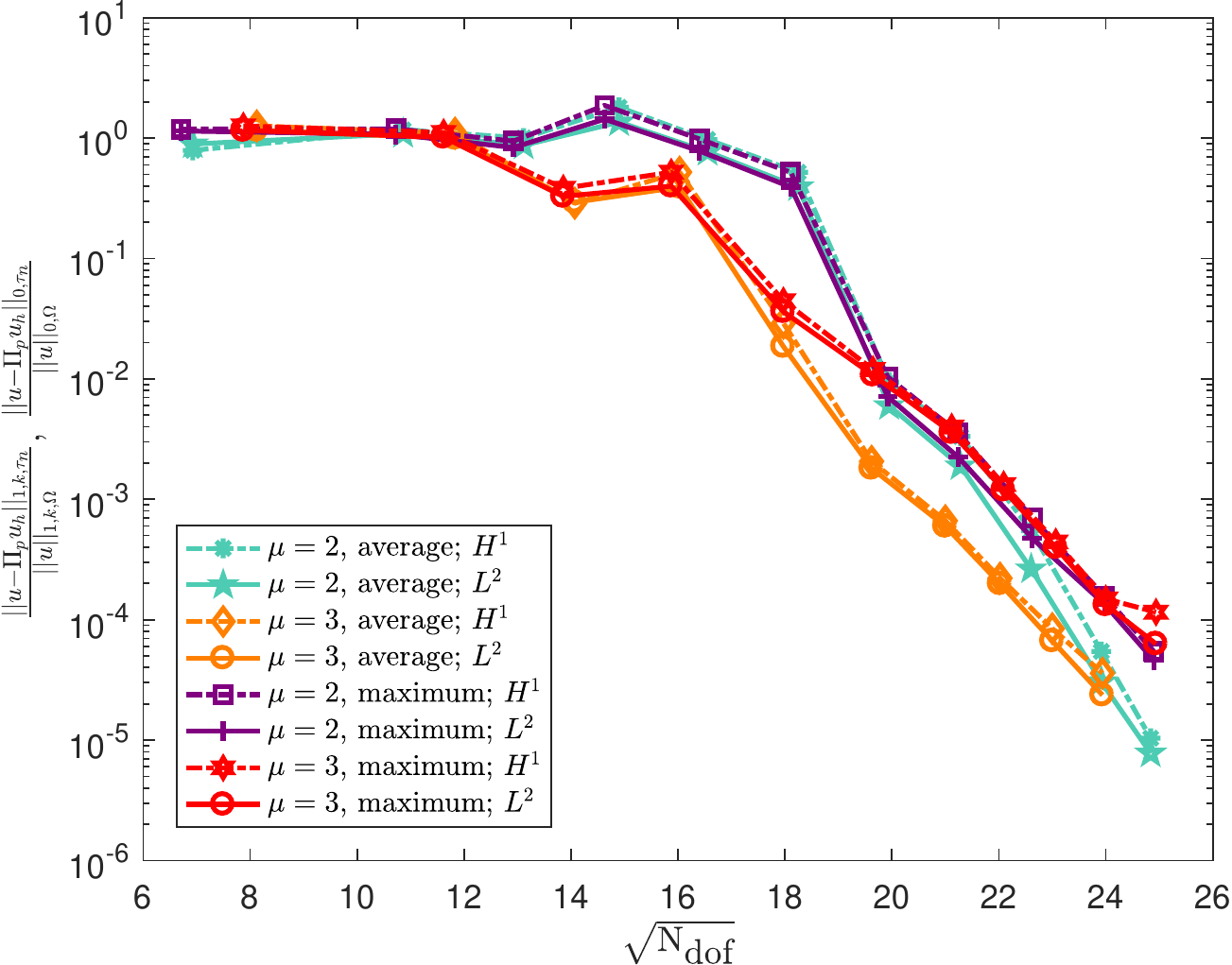}
\end{minipage}
\end{center}
\caption{$\h\p$-version with anisotropic geometrically graded meshes, with grading parameter~$\sigma$ in~\eqref{sigma ISO} equal to~$1/3$.
We consider the uniform distribution of the effective degrees~\eqref{uniform p} and compare the effects of the choice of the ``artificial'' wave number on the elements abutting the interface~$\Gamma$; in particular, we pick the average and the maximum of the two wave numbers.
On the $x$-axis, we plot the number of degrees of freedom; on the $y$-axis, we plot the computable relative~$H^1$ and $L^2$ errors in~\eqref{rel_errors}.}
\label{fig: different wn}
\end{figure}

From Figure~\ref{fig: different wn}, we deduce that the choice for the ``artificial'' wave number is not particularly influencing the method, although the performance, when picking the average, seems to be slightly better.

\section{Conclusions} \label{section conclusions}
We have extended the nonconforming Trefftz virtual element method of~\cite{ncTVEM_Helmholtz, TVEM_Helmholtz_num} for the approximation of solutions to Helmholtz boundary value problems
to the case of piecewise constant wave numbers, modelling fluid-fluid interface problems.
Moreover, we discussed the enrichment of the local approximation spaces with special functions, capturing the physical behaviour of the solution to the target problem.

Owing to the nonconforming setting of the method, and more precisely to the edgewise definition of the basis functions, this can be done in a natural fashion by simply supplementing the edge spaces with the corresponding traces of the functions.
Although this procedure results in a large number of degrees of freedom, an or\-tho\-go\-na\-li\-za\-tion-and-fil\-te\-ring process as introduced in~\cite{TVEM_Helmholtz_num}
can be applied to eliminate ``plonastic'' basis functions and mitigate the strong ill-conditioning, eventually leading to an extremely robust performance of the method.

This is verified in a number of numerical experiments, including investigations on~$\h$-, $\p$-, and $\h\p$-refinements.
In particular, whereas the~$\h$- and the $\p$-versions of the method converge optimally when employing meshes which are conforming with respect to the interface~$\Gamma$,
this is not the case anymore when some elements of the mesh are cut by~$\Gamma$: due to the low global Sobolev regularity of the solution to the fluid-fluid interface problem, the convergence rate is very poor.
Therefore, we resorted to the $\h\p$-version of the method using geometrically graded meshes in both an isotropic and an anisotropic fashion,
recovering algebraic and exponential convergence in terms of the number of degrees of freedom in the former and latter cases, respectively.

Lastly, we highlight that, although the method presented herein has been described for 2D problems only,
it can be generalized to the 3D case, as discussed in~\cite[Section 3.7]{ncHVEM} for the nonconforming harmonic VEM, with a minor effort.

\paragraph*{Acknowledgements}
The authors have been funded by the Austrian Science Fund (FWF) through the project F 65 (L.M.) and the project P 29197-N32 (A.P.).

{\footnotesize
\bibliography{bibliogr}
}
\bibliographystyle{plain}

\end{document}